\documentclass[12pt,reqno]{amsart}

\usepackage{amsfonts, amssymb, epsf, amsmath, graphicx, comment,xcolor}
\newcommand{\R}{\mathbb{R}}
\newcommand{\N}{\mathbb{N}}
\newcommand{\Q}{\mathbb{Q}}
\newcommand{\Z}{\mathbb{Z}}

\newcommand{\A}{\mathcal{A}}

\newcommand{\F}{\mathcal{F}}

\newcommand{\larr}{\left( \begin{array}{c}}
\newcommand{\rarr}{\end{array} \right) }

\newcommand{\lsqarr}{\left[ \begin{array}{c}} 
\newcommand{\rsqarr}{\end{array} \right]}

\newcommand{\be}{\begin{equation}}
\newcommand{\ee}{\end{equation}}
\newtheorem{theorem}{Theorem}

\newtheorem{cor}[theorem]{Corollary}

\newtheorem{lemma}[theorem]{Lemma} 
\newtheorem{prop}[theorem]{Proposition}
\theoremstyle{definition}

\newtheorem{example}{Example}
\begin{document}

\title[Rotation Number]{Rotation Numbers and Rotation Classes on One-Dimensional Tiling Spaces}
\author{Jos\'e Aliste-Prieto}\address{Departamento de Matem\'aticas, Universidad Andres Bello, Republica 498, Santiago, Chile}
\email{jose.aliste@unab.cl}
\author{Betseygail Rand}\address{Department of Mathematics, 
Texas Lutheran University, Seguin, TX 78155}\email{brand@tlu.edu}
\author{Lorenzo Sadun}\address{Department of Mathematics, 
University of Texas, Austin, TX 78712} \email{sadun@math.utexas.edu}
\maketitle
\begin{abstract}
We extend rotation theory of circle maps to tiling spaces. Specifically, we consider
a 1-dimensional tiling space $\Omega$ with finite local complexity and study 
 self-maps $F$ that are  homotopic to the
identity and whose displacements are strongly pattern equivariant (sPE).  
In place of the familiar rotation number we define a cohomology class 
$[\mu] \in \check H^1(\Omega, \R)$. We prove existence and uniqueness results for this class,
develop a notion of irrationality, and prove an analogue of Poncar\'{e}'s Theorem: If $[\mu]$ is irrational, then $F$ is semi-conjugate to uniform
translation on a space $\Omega_\mu$ of tilings that is homeomorphic to $\Omega$. In such cases, $F$ is semi-conjugate
to uniform translation on $\Omega$ itself if and only if $[\mu]$ lies in a certain subspace
of $\check H^1(\Omega, \R)$. 
\end{abstract}

\section{Introduction}

Since 1885, rotations numbers have been used to understand
orientation-preserving self-homemorphisms $f:S^1 \to S^1$ of the circle $S^1=\R/\Z$.  
The basic questions of rotation theory are:
\begin{enumerate}
    \item When is $f$ the time-1 sampling of a uni-directional flow on $S^1$?
    \item When is $f$ conjugate to a uniform rotation on $S^1$?
    \item When is $f$ semi-conjugate to a uniform rotation on $S^1$?
\end{enumerate}
The first two questions are of course equivalent, since any uni-directional flow is conjugate to a constant flow. However, in 
the more general setting considered in this paper, they will turn out to be different. 

Every orientation-preserving homeomorphism $f: S^1 \to S^1$ lifts to a homeomorphism $F: \R \to \R$ such that 
$F(x+1)=F(x)+1$ for all $x \in \R$. Poincar\'e \cite{Poincare1885} defined the {\em translation} number of $F$
to be $\lim_{n \to \infty} \frac{F^n(x)-x}{n}$. This limit exists for all lifts $F$ and
starting points $x$, and does not depend on $x$. Different lifts give 
translation numbers
that differ by integers, and the {\em rotation} number of $f$ is their common image in
$\R/\Z$.

Poincar\'e showed that, if the rotation number is irrational, then $f$ is 
semi-conjugate to a uniform rotation on the circle. Denjoy \cite{Denjoy} generated
examples where this semi-conjugacy is not a conjugacy. However, he also showed that 
if $f$ is sufficiently
smooth (specifically, if $f'$ has bounded variation, and in particular if $f$ is $C^2$) then the semi-conjugacy must
in fact be a conjugacy. 
If $L$ is any length, the circle $\R/L\Z$ can be viewed as the orbit,
under translations, of a tiling that is periodic with period $L$. 
We can then view an orientation-preserving homeomorphism of 
$\R/L\Z$ as a map on a space of periodic tilings. 
The goal of this paper, which can be viewed as an extension of \cite{aliste2010translation},
is to extend rotation theory to 1-dimensional {\em non\/}-periodic tiling spaces.

Extensions of rotation theory to higher dimensions and to non-periodic settings have been 
considered by many authors in the literature. Misiurewicz and Ziemian 
\cite{misiurewicz1989rotation} extended rotation theory to maps of tori that are homotopic
to the identity. There has been an extensive work on the study of such maps and a 
Poincar\'e Theorem was obtained first for quasiperiodically forced circle maps by Stark and
J\"ager in \cite{jager2006towards} and later extended to $\rho$-bounded pseudo-rotations 
by J\"ager in \cite{jager2009linearization}. Extending rotation-theory to non-periodic 
settings goes back perhaps to R.~Johnson and J.~M\"oser \cite{johnson1982rotation} who 
studied the rotation number of an almost periodic Schr\"odinger equation. J.~Kwapisz 
\cite{Kwapisz2000} studied the rotation sets for maps of the real line with almost-periodic 
displacement in the sense of Bohr. Tiling spaces have a similar structure to solenoids, 
and rotation theory for solenoids was developed by A.~Clark \cite{Clark05}. Later, the first author \cite{aliste2010translation} studied the rotation numbers for maps of the real-line with pattern-equivariant displacements and obtained a Poincar\'e Theorem for $\rho$-bounded maps with irrational rotation numbers. In \cite{aliste2012almost}, the first author and T. J\"ager obtained a Poincar\'e Theorem that includes at the same time the almost-periodic case studied by Kwapisz and the quasiperiodically forced circle case studied by Stark and J\"ager.

\subsection{Statements of the main theorems}

If $\Omega$ is such a tiling
space, we study
self-homeomorphisms $F: \Omega \to \Omega$ that are homotopic to the identity (``identity-homotopic'')
and that have  {\em strong pattern equivariant displacement} (sPE displacement) (See section 2 for a precise definition) .  
Instead of defining a rotation {\em number} $\rho$, we define a rotation {\em class}
$[\mu]$ in the first \v Cech cohomology $\check H^1(\Omega, \R)$, and define what it means for this class to be irrational. (Note that 
$\check H^1(S^1,\R)=\R$, so a rotation class for a circle map is just a number. The usual rotation number 
$\rho$ turns out to be $\mu^{-1} \pmod{1}$.) We prove the following main theorems under the assumptions that our tiling
space $\Omega$ is compact, minimal and uniquely ergodic, and
that $F: \Omega \to \Omega$ is a strongly pattern equivariant
map, homotopic to the identity. These terms are defined 
precisely in Section \ref{sec-defn}.
If these assumptions are met, we write $F \in \F(\Omega)$, or sometimes just $F \in \F$
when the space $\Omega$ is clear. 

\begin{theorem}[Proposition \ref{mu-exists}] \label{main1-1} If $F \in \F(\Omega)$ 
is the time-one sampling 
of a uni-directional strongly pattern-equivariant flow, and if $F$ has no fixed points,
then the rotation class $[\mu]$ of $F$ exists. \end{theorem}

In Section \ref{sec-examples}, we exhibit a tiling space $\Omega$ and a map $F \in \F$ with no fixed points 
such that $[\mu]$ does not exist, implying that this $F$ does not come from an sPE flow. 

Rotation classes are defined, using a de-Rham like cohomology theory, via differential forms ({\em rotation forms}) meeting certain combinatorial conditions. If $\mu$ meets
these conditions and $\mu'$ is cohomologous to $\mu$, then $\mu'$ also meets these conditions and
we say that $[\mu]$ is a rotation {\em class}.

\begin{theorem}[Theorem \ref{thm-unique}] If $F \in \F(\Omega)$ has rotation forms
$\mu$ and $\nu$, and if $[\mu]$ is irrational, then $\mu$ and $\nu$ are cohomologous.
\end{theorem}

That is, rotation classes are unique if they are irrational. However, there exist maps
that admit multiple rotation classes, all of them rational. We construct such a map
in Section \ref{sec-examples}.

\begin{theorem}[Theorem \ref{irr implies semiconj} and Corollary \ref{irr implies local semiconj}] \label{main2-1} If $F \in \F(\Omega)$ and if the rotation class 
$[\mu]$ of $F$ exists and is 
irrational, then $F$ is semi-conjugate to the time-one 
sampling of 
a strongly PE uni-directional flow on $\Omega$. \end{theorem}

There are two key differences between these results and the 
classical theory of circle maps. The first difference is that we do not have an 
analogue of Denjoy's Theorem. While we conjecture that
some version of Denjoy's Theorem is still true, the usual proofs fail spectacularly in the
setting of tiling spaces. 

The second difference is that if $F$ is (semi-)conjugate to 
the time-1 sampling of a uni-directional (and strongly PE) flow on $\Omega$, this does {\em not} imply that $F$ is (semi-)conjugate to a uniform translation
on $\Omega$ itself. Instead, it implies that $F$ can be (semi-)conjugated to a uniform translation on {\em another} tiling space $\Omega'$ that is homeomorphic to $\Omega$. Whether the translation action on $\Omega$ is topologically conjugate (up to a uniform rescaling) to the translation action on $\Omega'$ is a purely cohomological question.
Applying results of \cite{Julien-Sadun}, we show there is a subspace of $H^1(\Omega,\R)$, denoted $\R \, dx \oplus H^1_{AN}(\Omega,\R)$, such that the translation actions on $\Omega$ and $\Omega'$ are topologically conjugate, up to a uniform rescaling, if $[\mu]$ lies in that subspace. This implies that: 

\begin{theorem} \label{main4-1} Suppose that $F \in \F(\Omega)$, 
that $[\mu]$ exists and is irrational, and that $[\mu] \in \R \, dx \oplus 
H^1_{AN}(\Omega,\R) \subset H^1(\Omega,\R)$. Then $F$ is semi-conjugate to a uniform 
translation on $\Omega$. 
\end{theorem}

The converse to Theorem \ref{main4-1} is false. Some tiling spaces, such as those that come from 
substitutions, admit complicated self-homeomorphisms. If $F$ has irrational rotation class 
$[\mu]$ and $G: \Omega \to \Omega$ is such a homeomorphism, then
 $F'=G^{-1} \circ F \circ G$ has rotation class $[G^*\mu]$. If $F$ is semi-conjugate (by a map $J$) to
uniform translation, then $F'$ is also semi-conjugate (by $G^{-1} \circ J \circ G$) to uniform translation. However, it is possible to have $[\mu] \in \R \, dx \oplus H^1_{AN}(\Omega,\R) \subset H^1(\Omega,\R)$
without having $[G^*\mu]\in \R \, dx \oplus H^1_{AN}(\Omega,\R) \subset H^1(\Omega,\R)$. 
To avoid the difficulties posed by self-homeomorphisms of $\Omega$, we must restrict the form of the semi-conjugacy:

\begin{theorem} \label{main5-1} Suppose that $F \in \F(\Omega)$ and that $[\mu]$ exists and is irrational. $F$ is semi-conjugate to a uniform translation
on $\Omega$, via a
semi-conjugacy that sends each path component of $\Omega$ to itself, if and only if 
$[\mu] \in \R \, dx \oplus H^1_{AN}(\Omega,\R) \subset H^1(\Omega,\R)$.
\end{theorem}

Theorem \ref{main4-1} is of course a corollary of Theorem \ref{main5-1}. We will prove Theorem \ref{main5-1} in Section \ref{sec-Denjoy}.

Theorem \ref{main5-1} is a version of Poincar\'e's theorem that replaces the bounded-mean motion hypothesis on $F$ by a
cohomological condition that depends \emph{only} on $[\mu]$. In general, checking whether a given map has bounded mean motion 
is difficult, while the cohomological condition in Theorem \ref{main5-1} can often be checked with standard tools, 
especially when the tiling space comes from a substitution or a cut-and-project scheme.  The main challenge in applying this theorem lies in developing methods to compute the rotation class of a given map. 

\section{Background and precise definitions}\label{sec-defn}
\subsection{Tilings}

In one dimension, a tile is a pair $t=\{I,\ell\}$, where $I$ is a closed interval and $\ell$ is a label. We assume that the labels are
drawn from a finite {\em alphabet}, and that
tiles with the same label always have intervals of the same length. 
If $t = \{[x_1,x_2], \ell\}$ is a tile,
and if $s \in \R$, then we define $t-s = \{[x_1-s, x_2-s], \ell\}$. 
That is, translating a tile means moving its support without changing its label. 
A {\em tiling} is a collection of tiles, intersecting only at their boundaries, 
whose union is all of $\R$. If 
$T = \{t_i\}$ is a tiling and $s \in \R$, then $T-s = \{t_i-s\}$. We sometimes denote the
action of translations on tilings by $\Gamma$, so $\Gamma_s(T) := T-s$. 
Tilings whose tiles satisfy the above assumptions are said to have {\em finite local
complexity}, or FLC. \footnote{The FLC condition is usually defined in terms of local patches, but in one dimension it is equivalent to
simply having finitely many possible kinds of tiles.}

If $T$ and $T'$ are tilings built on the same set of possible tiles, consider the set of $\epsilon \in (0,1)$ 
such that there exist $s_{1,2} \in \R$
with $|s_i|  \le \epsilon/2$, and such that $T-s_1$ and $T-s_2$ agree exactly on the interval $[-\epsilon^{-1}, \epsilon^{-1}]$. We then define
the distance $d(T,T')$ to be the infimum of such $\epsilon$'s, or 1 if no such $\epsilon$ exists. That is, $d(T,T') \le \epsilon$ if 
$T$ and $T'$ agree on $[-\epsilon^{-1}, \epsilon^{-1}]$ up to translations by up to $\epsilon/2$. With the topology defined by this metric,
the set of all possible tilings by the fixed tile set is a compact space on which $\R$ acts by translation, i.e., an abstract dynamical system.

The {\em orbit} of a tiling $T$ is the set $\{T-s | s \in \R\}$. For a fixed tiling $T$, we often identify the orbit of $T$ with a copy of 
$\R$ by $T-s \leftrightarrow s$. The closure of the orbit of $T$ is called the {\em continuous hull} of $T$, or the {\em tiling space} of $T$, 
or simply the {\em orbit closure}, and is denoted $\Omega_T$. This is a dynamical system in its own right. A tiling $T'$ is in $\Omega_T$
if and only if every pattern that appears in $T'$ appears somewhere in $T$. 

We next consider the local topology of $\Omega_T$. If $T' \in \Omega_T$ and $d(T'',T')<\epsilon$, then $T''$ and $T'$ agree on a big ball around the origin, up to a small translation. A neighborhood of $T'$ is then determined by a small real number 
(of size $<\epsilon$) describing the translation, and a point in a totally disconnected space describing the possible extensions of the
tiling beyond the big ball. 

$T$ is said to be {\em repetitive} if, for each finite pattern $P$ that appears in $T$, there is a length $L_P$ such that every interval of 
length $L$ in $T$ contains at least one copy of $P$. This is equivalent to $\Omega_T$ being a minimal dynamical system. That is, if $T$ is repetitive and $T' \in 
\Omega_T$, then $\Omega_{T'} = \Omega_T$ and the set of patterns that appear in $T'$ (sometimes called the {\em language} of $T'$) 
is the same as the set of patterns that appear in $T$.  Since all the tilings in $\Omega_T$ have the same orbit closure, we usually denote their common orbit closure as $\Omega$, without any subscripts. In this case, the totally disconnected set described in the previous paragraph is 
actually a Cantor set. 

In this paper, we only consider tilings that are one-dimensional, have FLC, and are repetitive.

\subsection{Pattern-equivariant cohomology}

Let $T$ be a tiling, let $\phi: \R \to \R$ be a continuous function, and let $R>0$. We say that $\phi$ is {\em pattern equivariant (PE) with radius $R$} with respect to
$T$ if $\phi(x_1)=\phi(x_2)$ for all pairs $(x_1,x_2)$ of points such that $T-x_1$ and $T-x_2$ agree exactly on $[-R,R]$. That is, the value of $\phi$ at a point $x \in \R$ is determined exactly by the pattern of $T$ on $[x-R,x+R]$. A function $\phi$ is called {\em weakly PE} (wPE) if it is the uniform limit of 
PE functions, meaning that for any $\epsilon >0$ there exists an $R>0$ such that $\phi(x)$ is determined to within $\epsilon$ by 
the pattern of $T$ on $[x-R,x+R]$. Here, we will refer to PE functions as {\em strongly PE} (sPE) to distinguish them from wPE functions. With respect to the local product structure of $\Omega$, sPE functions are continuous in the $\R$ direction and locally constant in the Cantor direction, while wPE functions are merely continuous in both directions.


Using the identification of $\R$ with the orbit of a tiling, we extend the ideas of strong and weak pattern equivariance to functions
on $\Omega$. Let $T$ be a reference tiling. Suppose $g$ is a map on the orbit of $T$   whose {\em displacement} $\phi_T(x) : = g(T-x) - x$  is sPE (or wPE). Then $g$ extends by continuity to a map $G$ on all of $\Omega$, such that $G$ restricted to every reference tiling $T$ has sPE (or wPE) displacement.

In addition to functions with sPE displacement, we can also consider sPE differential forms, either on $\R$ with respect to a reference tiling $T$, or on $\Omega$. Since the derivative of a function with sPE displacement is an sPE  differential forms, we define the first PE 
cohomology\footnote{In the literature, ``pattern equivariant'' usually means sPE. The term ``PE cohomology'' 
was defined 
before the theory of wPE forms was developed.} of $T$ to be 
\be H^1_{PE,T} = \frac{\hbox{sPE 1-forms}}{d(\hbox{sPE functions})}.
\ee
Kellendonk and Putnam \cite{Kellendonk03, Kellendonk-Putnam06} (see also \cite{Sadun07}) proved that:
\begin{prop} \label{same-cohomology}
If $T$ is a repetitive FLC tiling, then $H^1_{PE,T}$ is isomorphic to the 
real-valued \v Cech cohomology
$\check H^1(\Omega_T,\R)$.
\end{prop}
It may happen that an sPE 1-form $\alpha$ is not the derivative of a function with sPE displacement, but is the derivative of a { wPE} function.
In that case we say that the class $[\alpha] \in H^1_{PE,T}$ is {\em asymptotically negligible}. (This does not depend on which representative
we pick for $[\alpha]$.) The asymptotically negligible classes, or more precisely the corresponding elements of $\check H^1(\Omega, \R)$, form a subspace of $\check H^1(\Omega, \R)$ denoted $H^1_{AN}(\Omega)$. Both $\check H^1(\Omega,\R)$ and $H^1_{AN}(\Omega, \R)$ are well-studied, and many techniques
are known for computing them. (See \cite{Sadun15} and references therein.)

\subsection{Self-homeomorphisms and related functions}

Let $\Omega$ be a tiling space and let $T\in \Omega$ be an arbitrary reference tiling. As always, we suppose that $\Omega$ is compact and minimal, or equivalently that $\Omega = \Omega_T$, where $T$ has FLC and is repetitive. Let $F: \Omega \to \Omega$ be a homeomorphism homotopic to the identity. $F$ must then take the path-component of $T$ to itself. 
However, the path-component of $T$ is the same as the orbit of $T$, so $F(T) = T - \Phi(T)$ for some continuous function $\Phi: \Omega \to \R$ (see \cite{kwapisz2010topological}). There must also be a homeomorphism $f_T: \R \to \R$ such that, 
for all $x \in \R$, $F(T-x) = T- f_T(x)$.
We furthermore define $\phi_T(x)=f_T(x)-x$. That is, $F$ specifies where a tiling goes, while the {\em displacement function} $\Phi$ specifies how far it moves. If we identify
the orbit of $T$ with $\R$ by $T-x \leftrightarrow x$, then $F$ becomes $f_T$ and $\Phi$ becomes $\phi_T$. We will typically use capital letters to define maps on $\Omega$ and lower case letters to define maps on $\R$.
When the reference tiling $T$ is clear, we will omit the subscripts on $f$ and $\phi$.

In this paper we only consider maps $F$ for which $\Phi$ is sPE. This then implies that for every 
$T \in \Omega$, $\phi_T$ is sPE with respect to $T$. We will denote the set of self-homeomorphisms of $\Omega$ with sPE displacement by $\mathcal{F}(\Omega)$, or by $\F$ for short. Note that if $\phi_T$ is sPE, then $\phi_T$ is necessarily bounded, so $f_T$ is orientation preserving. 

\subsection{Rotation numbers, rotation forms, and rotation classes}

The goal of rotation theory is to understand what happens when you iterate a self-homeomorphism many times. For a lift $F: \R \to \R$ of a circle map $f: S^1 \to S^1$ this is described by a {\em translation number}, often denoted $\rho$ and defined as 
\be \rho = \lim_{n \to \infty} \frac{F^n(x)-x}{n}. \ee 
The translation number indicates how far you go on average per unit time, or equivalently how long it takes to travel a
certain distance. 

Now suppose that $\rho > 0$. If $x_1 \in \R$ and $L$ is a positive integer, then $x_1$ and $x_2 = x_1+L$ correspond to the same point on $S^1$. If $n_\pm$ are integers such that 
$n_- < L/\rho < n_+$, then $F^{n_-}(x_1) < x_2 < F^{n_+}(x_1)$.  In other words, integrating $dx/\rho$ from $x_1$ to $x_2$ gives
sharp upper and lower bounds on how many iterations are required to bring $x_1$ up to and then past $x_2$. Note that these estimates {\em only} apply when $x_2-x_1 \in \Z$. If we imagine the universal cover of $S^1$
as being a periodic tiling with period 1, the estimates apply whenever the patterns around $x_1$ and $x_2$ agree out to distance 1.

For maps $F: \Omega \to \Omega$ of tiling spaces, the analogue of the translation number $\rho$ is easy. We say that $\rho$ is the {\em rotation number} of $F$ if, for every tiling
$T$ and starting point $x$, $\lim_{n \to \infty} \frac{f_T^n(x)-x}{n} = \rho$.
In \cite{aliste2010translation}, the first author showed:
\begin{theorem} \label{rho-exists}
Let $(\Omega,\Gamma)$ be a uniquely ergodic and minimal one dimensional tiling space. Suppose that $F$ is a homeomorphism in $\mathcal{F}$ without fixed points. Then, the rotation number $\rho$ of $F$ exists and is different from zero. 
\end{theorem}
Observe that the assumption of unique ergodicity is essential. In Section \ref{sec-examples} we exhibit 
a map on a tiling space that meets all of the assumptions of this theorem except
unique ergodicity, and for which $\rho$ does not exist. Also observe that in our current context, the assumption of $F$ not having fixed points is equivalent to $\Phi$ not having zeros and (since $\Omega$ is compact) being bounded away from zero. In the case that $\Phi$ has zeros, strong pattern equivariance implies that in each orbit the set of zeros of $\Phi$ is relatively dense and thus the rotation number exists and is equal to $0$.  

We now turn to the analogue of $dx/\rho$. Let $\mu$ be an sPE 1-form on $\R$ with respect to a reference tiling $T$. 
We say that $\mu$ is a {\em rotation form} for $F$
if there exists a radius $R$ such that, for all pairs of points $x_1, x_2 \in \R$ such that $T-x_1$ and $T-x_2$ agree 
on a ball of radius $R$ around the origin, 
and for all integers $n_\pm$ such that $n_- < \int_{x_1}^{x_2} \mu < n_+$, then $x_2$ is strictly 
between $f_T^{n_-} (x_1)$ and $f_T^{n_+}(x_1)$. This definition is crafted to apply even when 
$F$ moves points backwards and when $\int_{x_1}^{x_2} \mu$ is an integer. 
In the typical case where $f_T$ moves points forwards and $\int_{x_1}^{x_2} \mu$ is not an integer, the estimates simplify to  
\be f_T^{\lfloor \int_{x_1}^{x_2} \mu \rfloor}(x_1) < x_2 < f_T^{\lfloor  \int_{x_1}^{x_2} \mu \rfloor + 1}(x_1), \ee
where $\lfloor \int_{x_1}^{x_2} \mu \rfloor$ denotes the greatest integer less
than or equal to $\int_{x_1}^{x_2} \mu$.

If $\mu$ is a rotation form with radius $R$ and $\nu = \mu + dg$ is cohomologous to $\mu$, where $g$ is a function with sPE displacement and radius $R'$, then we claim that $\nu$ is a rotation form with radius $\max(R, R')$. This is because if
$x_1$ and $x_2$ are points such that $T-x_1$ and $T-x_2$ agree on a ball of radius $\max(R,R')$, then $\int_{x_1}^{x_2} dg = g(x_2) - g(x_1) = 0$, so $\int_{x_1}^{x_2} \nu = \int_{x_1}^{x_2} \mu$. In other words, 
every representative of the cohomology class $[\mu] \in H^1_{T, PE} \simeq \check H^1(\Omega, \R)$ is also a rotation form, and we say that $[\mu]$ is a {\em rotation class}. 

The first question about rotation classes is to understand their basic relationship with rotation numbers. The next statement gives a natural answer when the tiling space is minimal and uniquely ergodic.
\begin{prop} \label{what-is-rho}
Let  $(\Omega,\Gamma)$ be a uniquely ergodic and minimal one dimensional tiling space. Suppose $F\in\mathcal{F}_+$ has displacement with sPE radius $R$, and that $F$ has no fixed points. Suppose also that $\mu$ is a rotation form for $F$.
Let $T$ be a tiling in $\Omega$, and suppose that $\{x_n\}_{n\in\N}$ is a set of points in $\R$ such that the ball of radius $R$ around the origin in $T-x_{n}$ agrees with 
the ball of radius $R$ around the origin in $T$, and such that $\lim_{n \to \infty} x_n = \infty$. Then 
\be \rho = \lim_{n\rightarrow\infty} \frac{x_n}{\int_{0}^{x_n}\mu}. \label{Prop7-eq1}\ee
\end{prop}

\begin{proof}
We already know, by Theorem \ref{rho-exists}, that $\rho$ exists, is unique, and can be computed as 
\be \rho = \lim_{k\rightarrow\infty}\frac{f_T^k(0)}{k}.
\label{Prop7-eq2}\ee 
By the definition of $\mu$, if $k=\lfloor{\int_0^{x_n}\mu}\rfloor$, then 
$f_T^k(0)<x_n<f_T^{k+1}(0)$. But then 
\be \frac{f^k_T(0)}{k+1} < \frac{x_n}{\int_0^{x_n}\mu} < 
\frac{f^{k+1}_T(0)}{k}. \label{Prop7-eq3}\ee
Taking a limit of (\ref{Prop7-eq3}) as $n \to \infty$ 
(and therefore $k \to \infty$) and 
applying (\ref{Prop7-eq2}) then gives (\ref{Prop7-eq1}).
\end{proof}
We will see in Example \ref{norotationnumber} in Section 
\ref{sec-examples} that when the tiling space is not uniquely 
ergodic, it is possible for a map to have a rotation class but 
no rotation number. Going further,  neither the existence nor 
the uniqueness of the rotation class is obvious. 
In Section \ref{sec-Poincare}, we will show that, with the 
added assumption of irrationality, rotation classes are in fact
unique. However, in Section \ref{sec-examples} we will  
construct a map for which the rotation class does not exist. 

\subsection{Collaring, shape changes and irrationality}

Suppose that $T$ is a 1-dimensional tiling with a finite 
alphabet ${\mathcal A}$. We can generate a new tiling $T_c$ 
using the same intervals as the tiles in $T$, only with a 
larger alphabet. For each tile $t \in T$, we replace the label 
$\ell \in \A$ with a triple $(\ell_-)\ell(\ell_+)$, where 
$\ell$ is the original label of $t$ and $\ell_\pm$ are the 
original labels of the tiles preceding and following $t$. A 
tile equipped with labels indicating its predecessor and 
successor is called a {\em collared tile}, and the process of 
relabeling is called {\em collaring}. The new alphabet $\A_c 
\subset \A^3$ is still finite, so the new tiling $T_c$ still 
has FLC. Furthermore, $T_c$ is repetitive if and only if $T$ is
repetitive. The obvious forgetful map $\Omega_{T_c} \to 
\Omega_T$ is a homeomorphism that commutes with translation. 
Since the translation actions on $\Omega_T$ and $\Omega_{T_c}$ 
are conjugate, the spaces $\Omega_T$ and $\Omega_{T_c}$ are 
equivalent for our purposes. 
One can also repeat the collaring process. For any radius $R$, 
it is possible to collar enough times that the label of each 
(collared) tile indicates the pattern of (ordinary) tiles 
around it out to distance at least $R$. 

Now suppose that $\mu$ is a positive 1-form that is sPE with 
radius $R$ with respect to a tiling $T$. Without loss of 
generality, suppose that the tiles 
in $T$ have been collared out to that same distance $R$. We 
construct a new tiling $T'$, 
with the same alphabet as 
$T$, by preserving the labels of the tiles while moving each 
vertex $x$ of $T$ to position $\int_0^x \mu$. The new length of a tile $t$ with endpoints $a$ and $b$ is $\int_a^b 
\mu$. Since $\mu$ is sPE with radius $R$, and since the tile labels describe the pattern of $T$ out to distance $R$, this integral depends only on the label of $t$. 

Note that this shape change operation does not commute with translation. However, it does 
map the orbit of $T$ to the orbit of $T'$, and the orbit closure $\Omega_T$ to 
$\Omega_{T'}$. The space $\Omega'$ is called the {\em shape change of $\Omega$ by $\mu$}. 
For the theory of shape changes in 1 dimension and higher dimensions see 
\cite{Clark-Sadun03, Clark-Sadun06}.

If $\mu$ and $\nu$ are different positive sPE 1-forms, then we 
can compare the spaces $\Omega_\mu$ and $\Omega_\nu$ obtained 
by doing shape changes by $\mu$ and $\nu$, respectively. It turns out that the translation actions on these spaces are topologically conjugate if and only if $\mu - \nu$ is asymptotically negligible \cite{Clark-Sadun06}.
If $\mu$ and $\nu$ are cohomologous, then there is an even stronger equivalence between $\Omega_\mu$ and $\Omega_\nu$, called {\em Mutual Local Derivability}, or MLD. In short, $\Omega_\mu$ is determined up to MLD by $[\mu] \in \check H^1(\Omega, \R)$ and is determined up to topological conjugacy by the image of $[\mu]$ in the quotient space $\check H^1(\Omega, \R)/\check H^1_{AN}(\Omega, \R)$. \cite{BK10}

Shape changes by negative 1-forms are defined similarly, and are orientation reversing. If $\mu$ changes sign but has a nonzero average value (where averaging requires unique ergodicity of $\Omega_T$), then there is form $\bar \mu$, cohomologous to $\mu$, that does not change sign; we 
can then do a shape change by $\bar \mu$. When $\Omega_T$ is uniquely ergodic, we can thus do shape changes by all cohomology classes whose average values are nonzero \cite{Julien-Sadun}.

\subsection{$\rho$-boundedness, irrationality and Poincare-like theorems}
If $f$ is a circle homemorphism and its rotation number $\rho$ is irrational, then 
Poincar\'e's Theorem  asserts that $f$ is semi-conjugate to a rotation by $\rho$. 
In this paper, we wish to understand and state Poincar\'e-like theorems in terms of 
the irrationality of the rotation class of a tiling homeomorphism. 

Suppose that $[\mu]$ is a cohomology class represented by a 
positive (or negative) sPE cochain $\mu$. Let $\Gamma_1: 
\Omega_\mu \to \Omega_\mu$ be the operator of translation by 1. We say that $[\mu]$ is {\em irrational} if $\Omega_\mu$ is 
minimal with respect to the action of $\Gamma_1$. This 
condition only depends on the cohomology class of $\mu$, since 
different representatives of the same class give translation 
actions that are topologically conjugate. Equivalently, $[\mu]$ is irrational if none of the topological eigenvalues of the 
translation action on $\Omega_\mu$ lie in $\Q - \{0\}$. 

We also speak of a rotation number $\rho$ being irrational if $\Gamma_\rho: \Omega \to \Omega$ is minimal.\footnote{Note that $\rho$ is a dimensionful quantity, having units of 
length. As such, the naive definition of $\rho$ being a irrational {\em number}
does not make sense.} This is equivalent to $[dx/\rho]$ being an irrational class, since
if $\mu = dx/\rho$, then $\Gamma_1$ on $\Omega_\mu$ is conjugate to $\Gamma_\rho$ on
$\Omega$. 

We say that $F$ is {\em $\rho$-bounded} if, for every $T \in \Omega$ and every
starting point $x$, the sequence $\{f_T^n(x) -x - n\rho\}$ is bounded. This is equivalent
to the quantity $f_T^n(x)-x-n\rho$ being uniformly bounded as a function of $n$ and $x$.
The main result of \cite{aliste2010translation} is
\begin{theorem}
\label{poincare}
Let $(\Omega,\Gamma)$ be a uniquely ergodic and minimal one dimensional tiling space. Let $F:\Omega\rightarrow\Omega$ be a orientation preserving 
homeomorphism with irrational rotation number $\rho$. Suppose furthermore that $F$ is $\rho$-bounded. Then $F$ is semi-conjugate to $\Gamma_\rho$.
\end{theorem}

Thus, $\rho$-boundedness is a crucial property when dealing with Poincar\'e-like theorems. In practice, $\rho$-boundedness is rather difficult to check. Fortunately, the introduction of rotation classes allows us to understand $\rho$-boundedness in cohomological terms:
\begin{theorem}\label{theorem9}
Suppose that $F \in \F$ does not have fixed points. Let $\rho$ be the rotation number of $F$ and let $\mu$ be a rotation form for $F$. Then, the following assertions are equivalent:
\begin{enumerate}
    \item $F$ is $\rho$-bounded. 
    \item $\mu - \frac{dx}{\rho}$ is  asymptotically negligible. 
\end{enumerate}
\end{theorem}
\begin{proof}
Suppose that $\mu$ is a positive form, the case of negative
$\mu$ being similar. 
Let $\beta = \mu - \frac{dx}{\rho}$.  By the Gottschalk-Hedlund theorem (see \cite{Kellendonk-Sadun14} for a version of this 
theorem specifically adapted to tiling spaces), $\beta$ is 
asymptotically negligible if and only if its integral is 
bounded. More precisely, $\beta$ is asymptotically negligible 
if and only if the quantity 
\be I(x_1,x_2) :=\int_{x_1}^{x_2} \beta = \left ( \int_{x_1}^{x_2}\mu \right ) - \frac{x_2-x_1}{\rho}\ee
is bounded as a function of $x_1$ and $x_2$. 

First we show that $\rho$-boundedness of $F$ implies 
asymptotic negligibility of $\beta$. Let $R$ be the sPE radius
of $\mu$.
Suppose that $F$ is $\rho$-bounded. Then there exists a constant $C$ such that $|f_T^n(x)-x-n\rho|<C$ for all $n$. 
Let $x_1$ and $x_2$ be two real numbers, sufficiently far apart, and suppose without loss of generality that $x_1<x_2$. Let $x_3$ be such that (a) $T-x_3$ and $T-x_1$ agree on a ball of radius $R$ around the origin, (b) $x_1<x_3<x_2$, and (c) there is no other point between $x_3$ and $x_2$ satisfying properties (a) and (b). That is, $x_3$ is the largest return time, smaller than $x_2$, to the pattern of radius $R$ around $x_1$. By repetitivity and FLC, $|x_3-x_2|$ is uniformly bounded, implying that $I(x_3,x_2)$ is bounded. Since 
$I(x_1,x_2)=I(x_1,x_3)+ I(x_3,x_2)$, all that remains is to
bound $I(x_1,x_3)$.

To see that $I(x_1,x_3)$ is bounded, set $n=\lfloor \int_{x_1}^{x_3}\mu\rfloor$, so in particular 
\be n-1 < \int_{x_1}^{x_3} \mu < n+1. \ee 
Since $\mu$ is a rotation form, we must also have  
\be f^{n-1}(x_1)<x_3<f^{n+1}(x_1). \ee
Combining these estimates, we get 
\begin{eqnarray}
I(x_1,x_3) & = & \left (\int_{x_1}^{x_3} \mu\right ) - \frac{x_3 - x_1}{\rho} \cr 
& \le & n+1 - \frac{x_3-x_1}{\rho} \cr 
& < & 2 - \frac{f^{n-1}(x_1) - x_1 - (n-1)\rho}{\rho} \cr 
& \le & 2 + \frac{C}{\rho}.
\end{eqnarray}
Similarly, 
\begin{eqnarray}
I(x_1,x_3) & = & \left (\int_{x_1}^{x_3} \mu\right ) - \frac{x_3 - x_1}{\rho} \cr 
& \ge & n-1 - \frac{x_3-x_1}{\rho} \cr 
& > & -2 - \frac{f^{n+1}(x_1) - x_1 - (n+1)\rho}{\rho} \cr 
& \ge & -2 - \frac{C}{\rho}.
\end{eqnarray}

Now we show that asymptotic negligibility of $\beta$ implies $\rho$-boundedness of $F$. 
We need to show that there exists a $C>0$
such that
\be \left |f^n(x_1)-x_1-n\rho \right |<C\ee
for every natural number $n$ and every starting point $x_1$. 
For each pair $(x_1,n)$, let $x'=f^n(x_1)$, and let $R$ be the larger of the
PE radius of $f$ and that of the rotation form $\mu$. By repetitivity and finite local complexity, there is a radius $R'$ such that every ball of radius $R'$ contains at least one copy of every $T$-patch of radius $R$. In particular, we can
find a point $x_2$, with $|x_2-x'|<R'$, such that the $T$-patches of radius $R$ around
$x_1$ and $x_2$ agree. 

We will apply the triangle inequality several times.
First, since $|x'-x_2|$ is bounded (say, by $C_1$), we have 
\begin{eqnarray} 
\left |f^n(x_1)-x_1-n\rho \right | & = & \left | x' - x_1 - n\rho \right | \cr 
& \leq & \left |x'-x_2 \right |+|x_2-x_1-n\rho| \cr 
&<& C_1 + |x_2-x_1-n\rho|.\end{eqnarray}
Since $\beta$ is asymptotically negligible, $\int \beta$ is bounded by a constant $C_2$ and  
\begin{eqnarray} 
\left |x_2-x_1-n\rho \right | & \leq & \left |x_2-x_1-\rho\int_{x_1}^{x_2}\mu \right |+\rho \left |\int_{x_1}^{x_2}\mu-n \right | \cr & = &\left | \int_{x_1}^{x_2} \beta \right | +  \rho \left | \int_{x_1}^{x_2} \mu -n \right | \cr 
& < & C_2+\rho\left |\int_{x_1}^{x_2}\mu-n \right |.\end{eqnarray} 
To complete the proof, we need to bound 
$\left |\int_{x_1}^{x_2}\mu-n \right |$. Since $\mu$ is a rotation form for $F$ with radius at most $R$, we have
\be f^{k-1}(x_1)<x_2<f^{k+1}(x_1),\ee
where $k=\lfloor \int_{x_1}^{x_2} \mu \rfloor$.
This implies that $|x_2-f^k(x_1)|$ is bounded by the maximum value of $\phi$ (the displacement of $f$). Furthermore, $|x_2-f^n(x_1)|=|x_2-x'|$ is bounded by $R'$, so $|f^k(x_1)-f^n(x_1)|$ is bounded. Since $\phi$ has a positive minimum value, this in turn bounds $|k-n|$. 
But $k$ is within 1 of $\int_{x_1}^{x_2} \mu$, so we have bounded $\left | \int_{x_1}^{x_2} \mu -n \right |$.

\end{proof}

\section{Flows}\label{sec-flows}

In this section we consider a best-case scenario, when our map $F: \Omega \to \Omega$
is the time-1 sampling of an sPE flow on $\Omega$. In this situation, we explore
the remaining obstructions to $F$ being conjugate to uniform translation on $\Omega$. We also 
show that $F$ coming from a flow implies the existence of a rotation class. Thus the 
{\em non-}existence of a rotation class implies that $F$ does not come from such a flow, and in
particular is not conjugate to uniform translation. 

Let $\bar F: \Omega \times \R \to \Omega$ be a continuous map. We say that $\bar F$
is a {\em flow with sPE velocity} if there exists an sPE function $V: \Omega \to \R$ such
that, for every $(T,t) \in \Omega \times \R$, 
\be \bar F(T, t) = T - x_T(t), \quad x_T(0)=0, \quad \frac{dx_T(t)}{dt} = v_T(x_T(t)) :=  V(T-x_T(t)).\ee

For each $t \in \R$ we define the map $\bar F_t: \Omega \to \Omega$ by 
$\bar F_t(T) = \bar F(T,t)$. 
We say that $F$ is the {\em time-1 sampling} of $\bar F$ if $F= \bar F_1$.
To avoid trivialities, we assume that $F$ has no fixed points, or equivalently that $V$ has no zeros. 
As always, we use lower-case letters with subscript $T$ to 
denote maps on $\R$ (or $\R \times \R$) associated with maps on $\Omega$ (or $\Omega \times \R$) via a reference tiling $T$.
When the reference tiling $T$ is clear we drop the subscript $T$. In particular, 
$\bar f_t(x_0)$ is the solution to the differential equation
\be \frac{dx}{dt} = v(x); \qquad x(0)=x_0.\ee

\begin{prop}\label{mu-exists} Let $F \in \F$.
Let $T$ be a fixed reference tiling, and identify the orbit of $T$ with $\R$ via $T-x \leftrightarrow x$.
If the map $F$ is the time-1 sampling of an sPE flow $\bar F$ with velocity function $V$, then 
$\mu = dx/v(x)$ is a rotation cochain whose PE radius is the same as the PE radius of $v$. 
\end{prop}

\begin{proof} Suppose that $x_1$ and $x_2$ are points such that $T-x_1$ and $T-x_2$ agree on 
$B_R(0)$, where $R$ is the PE radius of $v$. Let $s=\int_{x_1}^{x_2} \frac{dx}{v(x)}$. Then 
$\bar f_s(x_1)=x_2$. Since the flow is unidirectional, if $s$ lies between two integers $n_-$ and $n_+$, then $x_2$ lies
between $\bar f_{n_-}(x_1) = f^{n_-}(x_1)$ and $\bar f_{n_+}(x_1) = f^{n_+}(x_1)$. 
\end{proof}

Let $\mu = dx/v(x)$, and let $\Omega_\mu$ be the tiling space obtained by applying the shape change
associated with $\mu$ to $\Omega$. The shape-change homeomorphism $S: \Omega \to \Omega_\mu$ conjugates $\bar F_t$ on 
$\Omega$ to $\Gamma_t$ (i.e., translation by $t$) on $\Omega_\mu$, and in particular conjugates $F = \bar F_1$ to 
$\Gamma_1$ on $\Omega_\mu$.
However, this does {\em not} imply that $F$ conjugates to a uniform translation on $\Omega$ itself! 
That depends on the cohomology class of $\mu$:

\begin{theorem} \label{GoodFlow} Let $\bar F: \Omega \times \R \to \Omega$ be a flow with never-zero sPE velocity function
$V: \Omega \to \R$, let $T$ be a reference tiling, and let $\mu = dx/v(x)$. Let $\rho$ be a nonzero real number. The following 
four conditions are then equivalent. 
\begin{enumerate}
    \item $\mu = \frac{dx}{\rho} + \beta$, where $\beta$ is asymptotically negligible.
    \item $F$ is $\rho$-bounded. 
    \item There is a homeomorphism $H: \Omega \to \Omega$, homotopic to the identity, 
    that conjugates $\bar F_t$ to $\Gamma_{\rho t}$ for every $t \in \R$.
    \item There is a homeomorphism $H: \Omega \to \Omega$, homotopic to the identity, that conjugates $F$ to $\Gamma_\rho$. 
\end{enumerate}
\end{theorem}

Note that we have already proven the equivalence of conditions (1) and (2) in Theorem \ref{theorem9}.
Before proving the rest of Theorem \ref{GoodFlow}, we state and prove a lemma relating conditions (3) and (4). 

\begin{lemma} \label{GoodEnough} Suppose $\Omega$ is the orbit closure of a 1-dimensional repetitive tiling
with FLC, that $\bar F$ is a flow on $\Omega$, generated by an 
sPE velocity function $V$, that $F=\bar F_1$, and that
$G: \Omega \to \Omega'$ is a homeomorphism that conjugates $F$ to $\Gamma_1$. 
Then there exists
a (possibly different) 
homeomorphism $\tilde G: \Omega \to \Omega'$ that conjugates $\bar F_t$ to
$\Gamma_t$ for all $t \in \R$.
\end{lemma}

\begin{proof}[Proof of Lemma \ref{GoodEnough}] 
Pick a reference tiling $T$, and let $X$ be the closure of the 
$F$-orbit of $T$. We will use the notation $T_t$, $T_s$, etc. as shorthand for $\bar F_t(T)$, 
$\bar F_s(T)$, etc. 
Let $S$ be the set of all $t$'s such that 
$T_t \in X$. We claim that $S$ is a 
closed subgroup of $\R$ that contains the integers. 

Closure comes from the fact that if there is a sequence of $t_i\in S$ 
converging to $t_\infty$, then $T_{t_\infty}$ is arbitrarily
well-approximated by the $T_{t_i}$'s, each of which can be arbitrarily
well-approximated by some $T_{n_i}=F^{n_i}(T)$, where $n_i\in\Z$, 
so $T_{n_\infty} \in X$ and $t_\infty\in S$. 

To see the group property, suppose that $s$ and $t$ are in $S$ and pick an 
$\epsilon>0$. We must find an integer $N$ such that $d(T_N,T_{s+t}) < \epsilon$, where $d$ is the tiling metric. 

The map $\bar F_t$ is uniformly continuous, so there is a $\delta_1$
such that if two tilings are within $\delta_1$, then their images under $\bar F_t$ are
within $\epsilon/2$. Since $s \in S$, we can find an integer $n$ such that $d(T_n,T_s)< \delta_1$. Since $F^n$ is uniformly continuous, there is a $\delta_2$
such that if two tilings are within $\delta_2$, then their images under $F^n$ are within
$\epsilon/2$. 
Since $t \in S$, there exist an integer $m$ such that $d(T_m, T_t)<\delta_2$.

We then have 
\begin{eqnarray}
d(T_{n+m}, F^n(T_t)) < \epsilon/2 & \hbox{since} & d(T_m, T_t)< \delta_2, \cr 
F^n(T_t) = & T_{n+t} =  & \bar F_t(T_n),  \cr 
d(\bar F_t(T_n), T_{s+t}) < \epsilon/2 & \hbox{since} & d(T_n, T_s) < \delta_1, \cr 
d(T_{n+m}, T_{s+t})<\epsilon & \hbox{by the} & \hbox{triangle inequality.}
\end{eqnarray}
Thus $s+t \in S$. Since $S$ is a closed subgroup of $\R$ containing $\Z$, either
$S = \R$ or $S$ is a cyclic group generated by a fraction $1/q$. We will treat each
of these cases in turn. 

First suppose that $S=\R$. 
For each $t \in \R$, let $\chi(t)$ be such that $G(T_t)=G(T)-\chi(t)$. 

Since $F=\bar F_1$ commutes with $\bar F_t$ and conjugates to $\Gamma_1$, we
must have $\chi(t+1)=\chi(t)+1$. We will show that for all $s,t \in \R$, 
$\chi(s+t)=\chi(s)+\chi(t)$. This, combined with $\chi$ being the identity
on $\Z$, implies that $\chi$ is the identity on $\Q$. By continuity, $\chi$
is then the identity on all of $\R$, and $G$ conjugates $\bar F_t$ to 
$\Gamma_t$ on the orbit of $T$. By continuity, $G$ then conjugates $\bar F_t$ to $\Gamma_t$ on all
of $\Omega$.

Since $S=\R$, we can find an integer $n$ such that 
$T_n$ approximates $T_t$
arbitrarily well. Since $\bar F_s$ is uniformly continuous, this implies that 
$T_{t+s}$ is approximated arbitrarily well by $T_{n+s}$. Now, for any tiling
$T'$, $G(T'_s)$ is a translate of $G(T')$, and the relative spacing of these two tilings
is a continuous function of $T'$. Thus the displacement between $G(T_{s+t})$ and 
$G(T_t)$, which is $\chi(t+s)-\chi(t)$, can be arbitrarily well approximated by 
the relative displacement of $G(T_n)$ and $G(T_{n+s})$, namely $\chi(n+s)-\chi(n)$. 

However, $G$ conjugates $F^n$ to translation by $n$, so $\chi(n)=n$ and $\chi(n+s)=n+\chi(s)$.
Thus $\chi(s+t)-\chi(t)$ is arbitrarily well approximated by $\chi(s)$, and must 
therefore equal $\chi(s)$. This completes the proof in the case that $S=\R$.

Finally, suppose that $S=\frac{1}{q}\Z$. The previous arguments show
that $\chi$ is the identity on $\frac{1}{q}\Z$, but say nothing about
how $\chi$ behaves on the interval $(0,1/q)$. However, no two points on this
interval are in the same orbit closure, so we are free to pick {\em any} 
orientation-preserving homeomorphism $\chi: [0, 1/q] \to [0,1/q]$. 
In particular, we can pick the identity, resulting in 
$\tilde G(T-t) = G(T)-t$ for all $t$. We then extend $\tilde G$ to all of $\Omega$
by continuity, and $\tilde G$ conjugates $\bar F_t$ to $\Gamma_t$.
\end{proof}

\begin{proof} [Proof of Theorem \ref{GoodFlow}]
The first two conditions are equivalent by Theorem \ref{theorem9}.
Next we show that the first and third conditions are equivalent. 
We apply a shape change by $\mu$ to $\Omega$ to get a new tiling space 
$\Omega_\mu$, and then apply a uniform dilation by $\rho$ to $\Omega_\mu$ to
get a further space $\Omega_{\rho \mu}$. Let $H_0: \Omega \to \Omega_{\rho \mu}$ be the resulting shape change map. 
These shape changes conjugate 
$\bar F_t$ on $\Omega$ to $\Gamma_t$ on $\Omega_\mu$ to $\Gamma_{\rho t}$ on 
$\Omega_{\rho \mu}$. In particular, $H_0$ conjugates $F$ on $\Omega$ to 
$\Gamma_\rho$ on $\Omega_{\rho \mu}$.

However, what we want is to find a homeomorphism $H: \Omega \to \Omega$, homotopic to the identity, that conjugates
$\bar F_t$ to $\Gamma_{\rho t}$. If this exists, then $H_0 \circ H^{-1}: \Omega \to \Omega_{\rho \mu}$ is a topological
conjugacy. Conversely, if there exists a topological conjugacy $G: \Omega \to \Omega_{\rho \mu}$, homotopic to $H_0$,
then we can take $H = G^{-1} \circ H_0$. Thus 
condition (3) is equivalent to the existence of a topological conjugacy $G$ between $\Omega$ and $\Omega_{\rho \mu}$ that
is homotopic to $H_0$.

The (non)existence of such a conjugacy is addressed with the results of \cite{Julien17, Julien-Sadun}. 
Each homeomorphism
$G: \Omega \to \Omega'$ of tiling spaces is associated to a class $[G] \in
\check H^1(\Omega, \R^d)$ (where $d$ is the dimension of the tiling space, in our
case $d=1$). Homotopic maps yield the same class, and the class of a shape change is represented by the form generating
the shape change, so if $\Omega'=\Omega_{\rho \mu}$ and $g$ is homotopic to 
$h_0$, then 
\be [G] = [H_0] = \rho \mu = dx + \rho \beta. \ee
By a theorem of \cite{Julien17}, 
the spaces $\Omega$ and $\Omega'$ are topologically conjugate,
via a map homotopic to $G$, if and only if $[G - dx]$ is asymptotically negligible.
Thus condition (3) is equivalent to the asymptotic negligibility of $\rho \beta$, which is
of course equivalent to the asymptotic negligibility of $\beta$.

Thus the first three conditions are equivalent. By specializing to $t \in \Z$, the 
third condition implies the fourth. Lemma \ref{GoodEnough}, applied to a uniform rescaling of $\Omega$ by $\rho^{-1}$, 
shows that the fourth condition implies the third. 
\end{proof}

\section{Existence, uniqueness, and Poincare's theorem.}\label{sec-Poincare}

    \subsection{Existence of $\mu$ and semi-conjugacy to a flow.}  
    
Proposition \ref{mu-exists} says that a map $F \in \F(\Omega)$ that comes from a flow with sPE velocity has an sPE rotation form $\mu$, and a corresponding rotation class $[\mu]$. But what
if $F$ is merely conjugate, or semi-conjugate, to a map $F'$ with sPE displacement on another space $\Omega'$
that comes from a flow with sPE velocity? Does that imply that $F$ itself comes from a flow with sPE velocity? If not,
does it at least imply that $F$ admits a rotation class?

These questions are surprisingly hard. The difficulty is that a general (semi-)conjugacy does not have
to respect the local product structures. If $H: \Omega \to \Omega'$ (semi-)conjugates a map
$F \in \F(\Omega)$ to a map $F' \in \F(\Omega')$, and if $\mu'$ is a rotation form for 
$F'$, there is no guarantee that $H^*\mu'$ is sPE. To make progress on these questions, 
we must restrict our attention to the subset of (semi-)conjugacies for which $H^* \mu'$ is 
in fact sPE.

Suppose that $F \in \F(\Omega)$, that $H: \Omega \to \Omega'$ is a map of tiling spaces,
and that $F'\in F(\Omega')$. If $F' \circ H = H \circ F$, then we say that $H$ is a {\em 
semi-conjugacy} of $F$ to $F'$. If in addition $H$ is a homeomorphism, we say that $H$
is a {\em conjugacy} of $F$ to $F'$. As usual, we identify the orbits of $T \in \Omega$ and 
$T'=H(T) \in \Omega'$ with $\R$, resulting in maps $f$ and $f'$ on $\R$, a map $h$ from
the copy of $\R$ representing the orbit of $T$ to the copy of $\R$ representing the orbit 
of $T'$, and 
a flow $\bar f': \R\times \R \to \R$. If $f'$ comes from an flow $\bar f'$ with sPE velocity function $v'$, and if the pullback by $h$ of the rotation 
form $\mu'=dx/v'(x)$ is sPE, then we say that $H$ is a {\em local} (semi-)conjugacy of $F$
to the time-1 sampling by a flow, and we say that $F$ is {\em locally (semi-)conjugate} to
$F'$. (The word ``local'' here refers to the fact that local semi-conjugacies
typically preserve the local product structure of the tiling space $\Omega$.)
    
These definitions may sound restrictive, but in fact we will show 
(Corollary \ref{irr implies local semiconj}) that whenever $F$ admits an {\em irrational} rotation class,
it is locally semi-conjugate to uniform motion. 

If $H$ is a local conjugacy, then we can pull $\bar F'$ back
to get a flow $\bar F$ on $\Omega$, such that $F$ is the time-1 sampling of $\bar F$. 
Since $\mu = H^*(dx/v') = dx/v$ is sPE, the function $v(x)$ is sPE, so 
$\mu = dx/v$ is a rotation cochain for $F$ by Proposition \ref{mu-exists}. The
following theorem shows that $F$ admits a rotation cochain
even when $H$ is only a local semi-conjugacy. 

\begin{theorem}\label{mu-exists-semi}Let $F \in \F(\Omega)$. Suppose that $F$ is locally semi-conjugate, via 
a map $H: \Omega \to \Omega'$, to the time-1 sampling of an sPE flow $\bar F'$ on $\Omega'$.
Pick reference tilings $T \in \Omega$ and $T'=H(T) \in \Omega'$, and identify the
orbits of $T$ and $T'$ with $\R$ as usual. Let $\mu'=dx/v'(x)$ and let $\mu=h^*(\mu')$. 
Then $\mu$ is an sPE rotation cochain for $F$ whose radius (as a rotation cochain) 
is the same as its PE radius. 
\end{theorem}

\begin{proof} Without loss of generality, supposed that
$\bar F$ is a forwards flow, so that $t_1 < t_2$ implies that 
$\bar f_{t_1}(y) < \bar f_{t_2}(y)$ for every $y \in \R$.

Let $R$ be the PE radius of $\mu$, and suppose that  
$T-x_1$ and $T-x_2$ agree out to radius $R$.
Then $\int_{h(x_1)}^{h(x_2)} \mu' $ is the travel time from $h(x_2)$ to $h(x_1)$ under the flow $\bar f'$. 
Suppose that $n_-, n_+ \in \mathbb{Z}$ and that 
$n_- < \int_{x_1}^{x_2} \mu < n_+$. Then
\begin{equation} n_- < \int_{h(x_1)}^{h(x_2)} \mu' < n_+.
\end{equation}
Thus, for any $y$,
\begin{equation}  \bar f_{n_-} (y) < \bar f_{\int_{h(x_1)}^{h(x_2)} \mu' } (y)  < \bar f_{n_+} (y),
\end{equation}
so
\begin{equation} 
\bar f_{n_-} (h(x_1))  < \bar f_{\int_{h(x_1)}^{h(x_2)} \mu' } (h(x_1))  < \bar f_{n_+} (h(x_1)), \end{equation} 
\begin{equation} \bar f_{n_-} (h(x_1))  <  h(x_2)  < \bar f_{n_+} (h(x_1)). \end{equation}
Since  $h \circ f = \bar f_1 \circ h$,
\begin{equation} h \circ f^{(n_-)}(x_1)
  <  h(x_2)  <  h \circ f^{(n_+)}(x_1). \end{equation} 
Note that although $x < y$ only implies $h(x) \leq h(y)$, 
it is still the case that $h(x) < h(y)$ implies $x < y$. Therefore
\begin{equation} f^{(n_-)}(x_1)
  <  x_2  <   f^{(n_+)}(x_1).
  \end{equation}
Since $x_1$ and $x_2$ were arbitrary points whose neighborhoods
agreed to distance $R$, $\mu$ is a rotation cochain with 
radius $R$. 
\end{proof}

\subsection{Uniqueness of the rotation class}

\begin{theorem}\label{thm-unique} Suppose that the map $F:\Omega \to \Omega$ has rotation forms
$\mu$ and $\nu$, and that $[\mu]$ is irrational. 
Then $\mu$ and $\nu$ are cohomologous. That is, the rotation {\em class} is unique.
\end{theorem}

\begin{proof} We do a shape change $S: \Omega \to \Omega_\mu$ 
by $\mu$. Using $S$ and $S^{-1}$, we transfer quantities from
$\Omega$ to $\Omega_\mu$. In the process, $\mu$ becomes 
$(S^{-1})^* \mu = dx$, 
$\nu$ becomes $\nu' = (S^{-1})^*(\nu)$, and $F$ becomes $F' = S\circ F
\circ S^{-1}$. By construction, $dx$ and $\nu'$ and 
are rotation forms on $\Omega_\mu$ for $F'$. We will show
that $dx$ and $\nu'$ are cohomologous on $\Omega_\mu$, which will imply
that $\mu$ and $\nu$ are cohomologous on $\Omega$. 

Let $T' \in \Omega_\mu$ be a reference tiling, and as usual
we identify the translational orbit of $T'$ with $\R$.
Let $x_1$ and $x_2$ be corresponding points in sufficiently large identical patches 
in $T'$. Since $dx$ is a rotation class, if there is
an integer $n$ such that $n < x_2-x_1 < n+1$, then $(f')^n(x_1) < x_2 < (f')^{n+1}(x_1)$.
Similarly, if $m < \int_{x_1}^{x_2} \nu' < m+1$ for some integer $m$, then 
$(f')^m(x_1) < x_2 < (f')^{m+1}(x_2)$. That is, $m$ must equal $n$, so whenever $x_2-x_1$ is 
not an integer, we must have 
$\lfloor \int_{x_1}^{x_2} \nu'\rfloor = \lfloor x_2-x_1 \rfloor$.

Since the integral of $\nu'-dx$ is bounded (by one) on returns to similar patches,
and since these patches appear with bounded gaps, the integral of $\nu'-dx$ is
bounded, so $\nu' - dx$ is asymptotically negligible. Thus there is 
a function $G$ with wPE displacement on $\Omega_\mu$, and a corresponding function $g$ with wPE displacement on $\R$, 
such that $\nu' = dx + dg$. Our goal is to show that $g$ has sPE displacement, hence that
$dx$ and $\nu'$ are cohomologous. 

If the displacement of $g$ is
not sPE, then there exists a positive $\epsilon$, and 
corresponding points $x_3$ and $x_4$ in identical large patches 
such that $g(x_4) - g(x_3) > \epsilon$. Since 
the action of $\Gamma_1$ is minimal on $\Omega_\mu$, we can find an integer $n$ such 
that $n+ x_3$
is less than $\epsilon/4$ to the right of a point $x_5$, such that
an arbitrarily large patch around $x_5$ agrees with the corresponding patch 
around $x_4$. Since $g$ is continuous, by making this patch large enough we can ensure that
$|g(x_5)-g(x_4)| < \epsilon/2$. Thus we have 
$g(x_5) \ge g(x_3) + \epsilon/2$ and $n - \epsilon/4 < x_5-x_3 < n$ for
some integer $n$. But then $\int_{x_3}^{x_5} \nu' =x_5-x_3 + g(x_5)-g(x_3)$ is slightly
{\em greater} than $n$, and the integer parts of $x_5-x_3$ and $\int_{x_3}^{x_5} \nu'$ do not
agree, which is a contradiction.

We conclude that $g$ does in fact have sPE displacement, so $\nu'$ and $dx$ are cohomologous, so $\mu$ and $\nu$ are cohomologous. 
\end{proof}

The assumption that $\mu$ is irrational cannot be removed. 
In Section \ref{sec-examples} we construct a map on a tiling
space that admits multiple rotation classes, all of them
rational. 

\subsection{Existence and irrationality imply semi-conjugacy}
    
\begin{theorem} Suppose that the map $F:\Omega \to \Omega$ has a rotation class $[\mu]$ that is irrational. Then $F$ is semi-conjugate to uniform translation by $1$ on $\Omega_\mu$.
\label{irr implies semiconj}
\end{theorem}

\begin{proof} The proof 
follows the same idea as the proof of Theorem \ref{thm-unique}. 
We start by doing  a shape change $S: \Omega \to \Omega_\mu$ 
by $\mu$. Pulling everything back to $\Omega_\mu$, we check that $\mu$ becomes $dx$, $F$ becomes $F' = S\circ F
\circ S^{-1}$ and that $dx$ is a rotation form for $F'$. Since $[\mu]$ is irrational, translation by $1$ on $\Omega_\mu$ is minimal, that is, $1$ is $\Omega_\mu$-irrational. 
By Theorem \ref{rho-exists}, $F'$ is $\rho$-bounded with $\rho=1$, which in turn implies that $F'$ preserves
orientation. Since $F'$ satisfies the assumptions of Theorem \ref{poincare} with $\rho=1$, there is a new map $H':\Omega_\mu\rightarrow\Omega_\mu$ 
which semi-conjugates $F'$ to translation by $1$. Hence $F$ is semi-conjugate, by $H' \circ S$, to translation by $1$ on $\Omega_\mu$. 
\end{proof}

Note that this theorem proves the existence of a semi-conjugacy, but does not prove
the existence of a {\em local} semiconjugacy. The results of this section
can be summarized as follows:
\begin{enumerate}
    \item If $\mu$ does not exist, then $F$ is not locally semi-conjugate to the time-1 sampling of an sPE flow.
\item If $\mu$ exists and is irrational, then $F$ is 
semi-conjugate to the time-1 sampling of an sPE flow. In the next section, we will see
that this semi-conjugacy is in fact local. 
\item If $\mu$ exists and is rational, there is no reason to expect $F$ 
to be semi-conjugate to the time-1 sampling of an sPE flow, or even of any flow. 
In fact, by modifying examples for 
circle maps, it is easy to construct examples of $F$'s that are not. 
\end{enumerate}

\section{Exploring the semi-conjugacy.}\label{sec-Denjoy}  

In this section we assume throughout that $\mu$ exists and is irrational, hence that $F$ 
is semi-conjugate to the time-1 sampling of a flow. We examine the possible 
form of this semi-conjugacy. First we show that the semi-conjugacy must be a local 
map. This allows for 
a strengthening of Theorem \ref{irr implies semiconj}, completing the proof of 
Theorem \ref{main2-1}. We then address Theorems \ref{main4-1} 
 and \ref{main5-1}. 
Finally, we consider the barriers to showing that 
the semi-conjugacy is actually a conjugacy. That is, we set up for framework for
the analogue of Denjoy's Theorem, and explain why the usual proofs for circle maps 
cannot be carried over to tilings. Stating and proving the analogue of Denjoy's Theorem
is an intriguing problem for future work. 

Let $F$ be a map on $\Omega$ with irrational  rotation form $\mu$. Applying a shape change by $\mu$, let $\Omega_\mu$ be the space whose rotation class is exactly $dx$. We will use the convention that maps denoted with $\tilde{\ }$ refer to $\Omega_\mu$, and so $\tilde{F}$ is a map on $\Omega_\mu$, where $\tilde{F}(T)=T-\tilde{\Phi}(T)$. As $\tilde F$ is iterated under $n$, let $\tilde \Phi^{n}$ be the displacement of $F^n(T)$, defined by $\tilde{F}^n(T)=T-\tilde{\Phi}^n(T)$. For $T \in \Omega_\mu$, let 
\begin{equation} 
\tilde{\Psi}(T)= \limsup_{n\in \N}\{\tilde \Phi^n(T)-n\}, \qquad \tilde{\jmath}(T)=T-\tilde \Psi(T),
\end{equation}
so that $\tilde{\Psi}:\Omega_\mu \rightarrow \R$ and $\tilde{\jmath}:\Omega_\mu \rightarrow \Omega_\mu$.

For a fixed reference tiling $\tilde{T} \in \Omega_\mu$  whose orbit is identified with $\R$, let $\tilde{f} = \tilde{F}|_{\tilde{T}}$ and $\tilde{f}(x)=x + \tilde{\phi}(x)$, as usual.
 For $x \in \tilde{T}$, let 
\begin{equation} 
\tilde{\psi}(x) = \limsup_{n\in \N} \{\tilde{\phi}^n(x) -n\}, \qquad \tilde{\jmath}(x)=x + \tilde{\psi}(x),
\end{equation} 
where $\tilde{\jmath}:\tilde{T} \rightarrow \tilde{T}$ and $\tilde{\psi}$ is its displacement. As with $\tilde \Phi^n$, $\tilde{\phi}^n(x)$ represents the displacement of $\tilde f^n(x)$, rather than the $n$th composition of $\tilde \phi$ with itself.

The proof of Theorem \ref{poincare} (see \cite{aliste2010translation}) shows that
 $\tilde{\Psi}$ and $\tilde{\jmath}$  are continuous on the orbit of $\tilde{T}$ , and that $\tilde{\jmath}$ extends by continuity to a semi-conjugacy.
Since  $\tilde{f}$ is monotonic, the function $\tilde{\jmath}$  is non-decreasing. If  $\tilde{\jmath}$  is strictly
increasing, then $\tilde{\jmath}$ is a conjugacy. However, if there are intervals where  $\tilde{\jmath}$ 
is constant, or equivalently where  $\tilde{\psi}$  has derivative $-1$, then $\tilde{\jmath}$ is not injective.
Distinguishing between these two cases boils down to understanding  $\tilde{\psi}$. 

\subsection{Strong pattern equivariance and locality}

\begin{theorem}
The function $\tilde{\psi}$ is sPE. That is, there exists a 
constant $R$ such that if $\tilde{T}-x$ and $\tilde{T}-y$ agree out to distance $R$, then $\tilde{\psi}(x) = \tilde{\psi}(y) $. \label{H-is-sPE}
\end{theorem}

\begin{proof} Let $R_0$ be the radius of the rotation form
$dx$. By the definition of a rotation form, if $\tilde{T}-x$ and $\tilde{T}-y$ agree out to distance $R_0$, and if 
$y-x$ is not an integer, 
then $(\tilde{f})^{L}(x) < y <(\tilde{f})^{L+1}(x)$, where $L = \lfloor y-x \rfloor$.  

Suppose that $\tilde{T}-x$ and $\tilde{T}-y$ agree to distance $R_0$, and that
$k$ is an integer such that 
$y > (\tilde{f})^k(x)$. Then
\begin{eqnarray} 
  (\tilde{f})^n(y) & >&   (\tilde{f})^{n+k}(x), \cr
  y + \tilde{\phi}^n(y) & > &  x + \tilde{\phi}^{n+k}(x), \cr
  \tilde{\phi}^n(y) -n & > &  \tilde{\phi}^{n+k}(x) - n - y + x, \cr
   & = &  \tilde{\phi}^{n+k}(x) - (n + k) - (y -k - x). 
 \end{eqnarray}
 As a result, 
 \begin{eqnarray} 
  \limsup_{n\in \N} \{(\tilde{\phi})^n(y) -n\} & > &  \limsup_{(n+k)\in \N} \{(\tilde{\phi})^{n+k}(x) -(n+k)\}  - (y -k - x), \cr
    \tilde{\psi}(y) & > &  \tilde{\psi}(x)  - (y-k-x). 
\end{eqnarray}
If $y-x$ is not an integer, then we can take $k=\lfloor y-x
\rfloor$, so 
\begin{equation}
    \tilde{\psi}(x) - \tilde{\psi}(y)< y - x - \lfloor y - x \rfloor. \label{H_T}
\end{equation}

The function $\tilde{\psi}$ comes from a continuous function on our 
tiling space $\Omega_\mu$, and so is wPE. If $\tilde{\psi}$ is not sPE, 
then there exist points $x_1$ and $x_2$ such that $\tilde{T}-x_1$ and 
$\tilde{T}-x_2$ agree out to $R_0$, and such that $\tilde{\psi}(x_1) \neq 
\tilde{\psi}(x_2)$. By swapping $x_1$ and $x_2$ if necessary, we can
assume that $\tilde{\psi}(x_1) = \tilde{\psi}(x_2) + \epsilon$, 
where $\epsilon > 0$. We will show that this is impossible. 

Since $\tilde{\psi}$ is wPE, there is a radius $R_1 > R_0$ such that if 
$\tilde{T}-y_1$ and $\tilde{T}-y_2$ agree out to $R_1$, then $\tilde{\psi}(y_1)$ and 
$\tilde{\psi}(y_2)$ are within $\epsilon/3$ of each other.

Let $V_1$ be the set of $x$'s such that $\tilde{T}-x$ agrees with $\tilde{T}-x_1$ out to distance $R_1$, and let $V_2$ be the set of $x$'s such that $\tilde{T}-x$ agrees with $\tilde{T}-x_2$ out to $R_1$. Then for every $y_1 \in V_1$ and $y_2 \in V_2$, $\tilde{\psi}(y_1)$ is within  $\epsilon/3$ of $\tilde{\psi}(x_1)$ and  $\tilde{\psi}(y_2)$ is within  $\epsilon/3$ of $\tilde{\psi}(x_2)$. Since $\tilde{\psi}(x_1)-\tilde{\psi}(x_2)=\epsilon$,  
we must have $\tilde{\psi}(y_1) - \tilde{\psi}(y_2) > \epsilon/3$.

Since $dx$ is irrational, the set of integer translates of 
$\tilde{T}-y_1$ is dense in $\Omega_\mu$. In particular, there 
is an integer $k$ such that $\tilde{T}-(y_1+k)$ agrees with 
$\tilde{T}-(x_2 - \frac{\epsilon}{6})$ on a ball of radius at least 
$R_1 + \frac{\epsilon}{6}$, 
up to a translation by less than
$\frac{\epsilon}{6}$. That is, there is a point $y_2 \in V_2$
such that 
\begin{equation}
    y_1 +k < y_2 < y_1 + k + \frac{\epsilon}{3}.
\end{equation}
However, $\tilde{T}-y_1$ and $\tilde{T}-y_2$ agree to distance at least 
$R_0$, so $\tilde{\psi}(y_1) - \tilde{\psi}(y_2) < \frac{\epsilon}{3}$ by equation 
(\ref{H_T}), which is a contradiction.
\end{proof}

\begin{cor}
Suppose that the map $F:\Omega \to \Omega$ has a rotation class $[\mu]$ that is irrational. Then $F$ is locally semi-conjugate to uniform translation by $1$ on $\Omega_\mu$.
\label{irr implies local semiconj}
\end{cor}

\begin{proof}
By Theorem \ref{irr implies semiconj}, $F$ is semi-conjugate to uniform translation
on $\Omega_\mu$. 
By Theorem \ref{H-is-sPE}, this semi-conjugacy is the composition of a shape change and
translation by the sPE function $\tilde{\Psi}$. Since both shape changes and translations by sPE functions
preserve the local product structure of tiling spaces, the pullback of $\tilde{\mu}=dx$ by the 
composition of these maps is an sPE form, so our semi-conjugacy is local. 
\end{proof}

\subsection{Proof of Theorems \ref{main4-1} and \ref{main5-1}}
We now turn to the proofs of Theorems \ref{main4-1} and \ref{main5-1}, which we restate for easy reference.

\begin{theorem} [Theorem \ref{main4-1}] Suppose that $F \in \F(\Omega)$, 
that $[\mu]$ exists and is irrational, and that $[\mu] \in \R \, dx \oplus 
H^1_{AN}(\Omega,\R) \subset H^1(\Omega,\R)$. Then $F$ is semi-conjugate to a uniform 
translation on $\Omega$. 
\end{theorem}
\begin{theorem}[Theorem \ref{main5-1}] Suppose that $F \in \F(\Omega)$ and that $[\mu]$ exists and is irrational. $F$ is semi-conjugate to a uniform translation
on $\Omega$, via a
semi-conjugacy that sends each path component of $\Omega$ to itself, if and only if 
$[\mu] \in \R \, dx \oplus H^1_{AN}(\Omega,\R) \subset H^1(\Omega,\R)$.
\end{theorem}

Both proofs rely on the following
result of \cite{Julien-Sadun}.
\begin{lemma} \label{JS18-lemma}
Let $S: \Omega \to \Omega_\mu$ be a shape change map according to the sPE form $\mu$ on $\Omega$. There is a map,
homotopic to $S^{-1}$, that conjugates a fixed translation on $\Omega_\mu$ to a fixed translation on $\Omega$
(possibly by a different distance) if and only if $[\mu] \in \R\,dx \oplus H^1_{AN}(\Omega,\R)$. 
\end{lemma}

\begin{proof}[Proof of Theorem \ref{main4-1}]
Suppose that $[\mu] \in \R\,dx \oplus H^1_{AN}(\Omega,\R)$. We have already established that $F$ is semi-conjugate to 
uniform translation on $\Omega_\mu$. However, by Lemma \ref{JS18-lemma}, 
uniform translation on $\Omega_\mu$ is conjugate to uniform translation on 
$\Omega$ itself, so $F$ is semi-conjugate to uniform translation on $\Omega$. \end{proof}

\begin{proof}[Proof of Theorem \ref{main5-1}]
We begin with the ``if'' part of Theorem \ref{main5-1}.  
The map $\tilde{J}$, constructed above, maps each path component of $\Omega_\mu$ to itself and semiconjugates
$\tilde F = S \circ F \circ S^{-1}$ to uniform translation on $\Omega_\mu$. 
If $[\mu] \in \R\,dx \oplus H^1_{AN}(\Omega,\R)$, then uniform translation on $\Omega_\mu$ is conjugate
to uniform translation on $\Omega$, via a homeomorphism $S_1^{-1}$ that is homotopic to $S^{-1}$. The map 
$J := S_1^{-1} \circ \tilde{J} \circ S: \Omega\to \Omega$ then sends each path component of 
$\Omega$ to itself and semi-conjugates $F$ to uniform motion on $\Omega$.

The ``only if'' direction is more difficult. By assumption, $[\mu]$ exists and is irrational and 
$F$ is semi-conjugate to a translation map on $\Omega$. We also know that $\tilde F$ is semi-conjugate to
a translation map on $\Omega_\mu$. We will show that the translations on $\Omega$ and 
$\Omega_\mu$ are conjugate and use Lemma \ref{JS18-lemma} to show that $[\mu] \in \R\,dx \oplus H^1_{AN}(\Omega,\R)$.

Let $J: \Omega \to \Omega$ be a map that semi-conjugates $F$ to a fixed translation on $\Omega$. That fixed translation
is of course the time-1 sampling of a (uniform translation) flow $K_t$ on $\Omega$. 
Let $\tilde K_t = S \circ K_t \circ S^{-1}$, and let $\tilde{J}' = S \circ J \circ S^{-1}$. Meanwhile, let $\tilde{J}$ be 
the $\limsup$ map for $\tilde F$. We then have that $\tilde{J}$ semi-conjugates $\tilde F$ to translation by 1, while
$\tilde{J}'$ semi-conjugates $\tilde F$ to $\tilde K_1$. Let $\tilde \Psi$ and $\tilde \Psi'$ be the displacements of 
$\tilde{J}$ and $\tilde{J}'$, respectively.
Picking a reference tiling $\tilde T$, we similarly consider the induced maps $\tilde f$, $\tilde{\jmath}$, 
$\tilde{\jmath}'$, $\tilde \psi$, and $\tilde \psi'$.

\begin{lemma} The map $\tilde{\jmath}'$ factors through $\tilde{\jmath}$. 
That is, if $x_1, x_2 \in \R$ and 
$\tilde \jmath(x_2) = \tilde \jmath(x_1)$, then $\tilde \jmath'(x_2)=\tilde\jmath'(x_1)$.
\end{lemma}

\begin{proof}[Proof of lemma]  We can assume without loss of 
generality that $x_1<x_2$ (since if $x_1=x_2$ there is nothing to prove) and 
that $\tilde \jmath(x_2) = \tilde \jmath(x_1)$. That is, 

\begin{eqnarray}
\limsup \left ( \tilde f^n(x_1) - n \right ) & = & x_1 + \limsup \left ( \tilde \phi^n(x_1) - n \right ) \cr 
& = & \tilde j(x_1) \cr & = &  \tilde j(x_2) \cr & = &  x_2 + \limsup \left ( \tilde \phi^n(x_2) - n\right ) \cr 
& = & \limsup \left ( \tilde f^n(x_2) - n\right ).
\end{eqnarray}

Since $\tilde f$ is orientation-preserving, for every $n$ we have $\tilde f^n(x_1) < \tilde f^n(x_2)$, and hence 
$\tilde f^n(x_1) - n < \tilde f^n(x_2)-n$. Pick $\epsilon > 0$. Since $\tilde f^n(x_1) - n$  is within $\epsilon/2$ of 
$\tilde \jmath(x_1)$ infinitely often, and since $\tilde f^n(x_2) - n$  is eventually bounded by 
$\tilde \jmath(x_1)+ \epsilon/2$, there are (infinitely many) values of $n$ for which  $\tilde f^n(x_2) - \tilde f^n(x_1)
< \epsilon$.

Since $\tilde \jmath'$ is continuous, this means that there exist values of $n$ for which 
$\tilde{\jmath}'((\tilde f)^n(x_1))$ and $\tilde{\jmath}'((\tilde f)^n(x_2))$ are arbitrarily close, 
so the time it takes to flow by $\tilde k_t$ from $\tilde{\jmath}'((\tilde f)^n(x_1))$ to 
$\tilde{\jmath}'((\tilde f)^n(x_2))$ is arbitrarily small. However, this is exactly the same as the 
time it takes to flow from $\tilde{\jmath}'(x_1)$ to $\tilde{\jmath}'(x_2)$, since 
$\tilde{\jmath}'$ semi-conjugates $\tilde f$ to $\tilde k_1$. Thus we must have $\tilde{\jmath}'(x_1)
= \tilde{\jmath}'(x_2)$.

\end{proof}

We return to the proof of Theorem \ref{main5-1}. Since $\tilde {\jmath}' $ (or $\tilde J'$) identifies all points that are identified by $\tilde {\jmath}$ (or $\tilde J$), there must be a 
continuous $\tilde H: \Omega_\mu \to \Omega_\mu$ such that  $\tilde J' = \tilde H \circ \tilde J$, and 
a corresponding map $\tilde h: \R \to \R$. To show that $\tilde H$ is a homeomorphism, we must show that 
$\tilde J'$ only identifies points that are identified by $\tilde J$. 

We now let $x_1 < x_2$ be points such that $\tilde{\jmath}(x_1)< \tilde{\jmath}(x_2)$. Let $y_1=\tilde{\jmath}(x_1)$ and $y_2=\tilde{\jmath}(x_2)$. 
Note that $\tilde h$ semi-conjugates $\tilde \Gamma_1$ (that is, translation by 1 on $\Omega_\mu$) to 
$\tilde k_1$, so
for all integers $n$ and $y \in \R$, 
\begin{equation} \tilde h(y+n) = \tilde k_n(\tilde h(y)).
\end{equation}
If $\tilde h$ identifies $y_1$ and $y_2$, $\tilde h$ must also identify $y_1+n$ and $y_2+n$. However, 
by the irrationality of $dx$, $\{\tilde T - (y_1+n)\}$ is dense in $\Omega_\mu$. Since $\tilde H$ is continuous, this 
means that $\tilde h$ must identify every point $y$ with $y+(y_2-y_1)$, and so must collapse the entire real line
to a point. Since this is a contradiction, we conclude that $\tilde h(y_2) \ne \tilde h(y_1)$. Therefore  $\tilde{\jmath}(x_1) = \tilde{\jmath}(x_2)$ iff $\tilde{\jmath}'(x_1) = \tilde{\jmath}'(x_2)$. That is, $\tilde H$ is
a homeomorphism. 

Since $\tilde H$ conjugates $\tilde \Gamma_1$ to $\tilde K_1$, and since 
$S^{-1}$ conjugates $\tilde K_1$ to $K_1=\Gamma_1$, $S^{-1} \circ \tilde H$ conjugates uniform translation on $\Omega_\mu$ to uniform translation on $\Omega$. Since $\tilde H$ preserves 
translational orbits,  $\tilde H$ is homotopic to the identity, so $S^{-1} \circ \tilde H$ is homotopic to $S^{-1}$. By
Lemma \ref{JS18-lemma}, $[\mu] \in  \R\,dx \oplus H^1_{AN}(\Omega,\R)$. \end{proof}

\subsection{Steps towards Denjoy's Theorem}

Ideally, we would like to prove an analogue of Denjoy's Theorem, that if our map
$F$ is smooth enough (say, $f_T$ being $C^2$ for each reference tiling $T$), and if 
$[\mu]$ is irrational, then $F$ is conjugate, and not merely semi-conjugate, to 
uniform translation on $\Omega_\mu$. 

Recall the usual proof of Denjoy's Theorem for circle maps. If $f: S^1 \to S^1$ is an 
orientation-preserving homeomorphism with irrational winding number, then there is a map 
$\Psi: S^1 \to S^1$ that semi-conjugates $f$ to an irrational rotation. The map $\Psi$ may collapse
intervals to points. Such intervals are called {\em wandering}. If $U$ is a wandering interval, then
$U_n := f^n(U)$ is also a wandering interval. If $n \ne m$ and $U_n \cap U_m \ne \emptyset$, then we must have $U_n = U_m$, in which case the left endpoint of $U_n$ is a periodic point of $f$, which contradicts
the irrationality of the winding number. Thus the intervals $\{U_n\}$ are disjoint. Since the total 
length of the intervals is bounded by the circumference of $S^1$, the length of $U_n$ must go to
zero as $n \to \pm \infty$ and the ratio $\frac{|U_0|^2}{|U_n| |U_{-n}|}$ must go to $\infty$. 
However, that ratio is bounded by a quantity
called the {\em distortion}. If $f$ is $C^2$ and $f'$ is never zero, then the distortion is in turn  
bounded by
the integral of $|f''/f'|$ over the entire circle. This is a contradiction, so wandering intervals cannot
exist, so $\Psi$ must be a genuine conjugacy.\footnote{When $f$ is $C^2$ but $f'$ is zero at isolated points
the argument is technically more complicated, but follows the same overall strategy.}

Theorem \ref{H-is-sPE} is a very powerful tool, since it says that the intervals that
are collapsed by $\tilde{\jmath}$ (and play a role similar to wandering intervals) can be identified 
locally as the intervals where $\tilde \psi$ has
derivative $-1$. If $R$ is the greater of the sPE radius of $\tilde f$ and the sPE radius of 
$\phi$, then we can consider an approximant $X_R$ to $\Omega_\mu$ obtained by collaring out
to distance $R$, and a map from $X_R$ to itself induced by $\tilde F$. Since $\tilde \Psi$ is sPE, it can be viewed
as a function 
of on $X_R$. We call the intervals where $\tilde \psi' = -1$ {\it collapsing intervals}. As with circle maps, our map induced by $\tilde F$ sends each collapsing interval onto another. 

Unfortunately, the resemblance to circle maps ends there. Relative to a reference tiling $T$, 
let $\hat U_0 \subset \R$ be an interval collapsed by $\tilde{\jmath}$, let $\hat U_n = (\tilde f)^n(\hat U_0)$,
and let $U_n$ be the image of $U_n$ in $X_R$. Of course the intervals $\hat U_n$ are disjoint, but
they march off to infinity, so there is no a priori reason why their total length should be finite. 
As for the collapsing intervals $U_n$, they don't have to be disjoint! If $U_n = U_0$,
this merely says that the patches around $\hat U_n$ and $\hat
U_0$ are the same. It does not imply that 
$\hat U_{2n}$ will be similar, or that we have a periodic orbit in $X_R$, and does not contradict the 
irrationality of $\mu$. $X_R$ is a {\em branched} manifold, and the map induced by $\tilde f$ on $X_R$ is not 
single-valued. When we reach branch points, it can pick either branch, and does so in a way that reflects the
underlying non-periodicity of our tilings.  

This is not to say that Denjoy's Theorem is false for tiling spaces. We conjecture that, when $f$ is $C^2$
and $\mu$ is irrational, $F$ will indeed be conjugate to uniform motion. However, the proof of this conjecture
will not be straightforward. It will either require properties of strong pattern equivariance such as stronger
versions of Theorem \ref{H-is-sPE}, or it will require a strategy qualitatively different from the usual
proof for circle maps.

\section{Examples}\label{sec-examples}

In this section we examine a number of different tiling spaces and maps on these
spaces to get a feel for the meaning of our main theorems.

\begin{example}[The Fibonacci tiling]

The Fibonacci tiling has alphabet $\A = \{a,b\}$ and the allowed (bi-infinite) 
sequences of tile labels are generated by the substitution $a \to ab$, $b \to a$.
We pick two arbitrary positive constants $L_a$ and $L_b$ to be the lengths of the
$a$ and $b$ tiles, and denote by $\Omega_{Fib}$ the resulting space of tilings. 

The cohomology of $\Omega$ does not depend on $L_a$ and $L_b$. Regardless of the 
tile lengths, $\check H^1(\Omega_{Fib},\R)=\R^2$.  Let $i_a$ be an sPE 1-form that
integrates to 1 on each $a$ tile and to 0 on each $b$ tile, and let $i_b$ integrate
to 1 on each $b$ tile and to 0 on each $a$ tile. Since the sequences associated to 
the Fibonacci tiling are Sturmian, the form $i_a - \phi i_b$, where
$\phi = (1+\sqrt{5})/2$, is asymptotically negligible. Meanwhile, $dx$ is 
cohomologous to $L_a i_a + L_b i_b$. 

Since $[dx]$ and $[i_a-\phi i_b]$ span $\check H^1(\Omega_{Fib}, \R)$,  every 
sPE 1-form can be written as a multiple of $dx$ plus something asymptotically
negligible. In particular, if $\bar F$ is {\em any} unidirectional flow on $\Omega$
with sPE velocity function $V$, and if $\mu = dx/v$, then there is a nonzero constant
$\rho$ and an asymptoticaly negligible form $\beta$ such that $\mu = \rho dx + \beta$.
By Theorem \ref{GoodFlow}, $\bar F$ is then topologically conjugate to uniform 
translation on $\Omega_{Fib}$ at speed $\rho$.

The key points are that $\check H^1(\Omega_{Fib},\R)/\check H^1_{AN}(\Omega_{Fib},\R) = \R$ and that this group is generated by $[dx]$. This is reminiscent of the situation for circle 
maps, where $H^1(S^1,\R)=\R$ is already generated by $[dx]$, and there are no
nontrivial asymptotically negligible classes. For any tiling space (meeting 
the usual assumptions) where $\check H^1/\check H^1_{AN}=\R$, an sPE 
self-homeomorphism $F$ that
is homotopic to the identity is conjugate to a uniform translation if and only if 
it is the time-1 sampling of an sPE flow.
\end{example}

\begin{example}[A non-Pisot substitution]

The situation is different when the substitution matrix is not Pisot. 
Consider a space $\Omega_{nP}$ of tilings generated by 
the substitution $a \to ab$, $b \to aaa$, with tile lengths $L_a=L_b=1$. 
As with the Fibonacci
tiling, $\check H^1(\Omega_{nP},\R)=\R^2$ is generated by $[i_a]$ and $[i_b]$. 
However, the substitution matrix 
$M = \left ( \begin{smallmatrix} 1 & 3 \cr 1 & 0 \end{smallmatrix} \right )$ has eigenvalues 
$( 1 \pm \sqrt{13})/2$, and hence is expansive. Any (nontrivial) 
linear combination of $i_a$ and $i_b$ gives large results when integrated over supertiles of 
sufficiently large order.  Thus $\check H^1_{AN}(\Omega_{nP},\R)$ is trivial. 

Suppose that we have n flow $\bar F$ on $\Omega_{nP}$ with sPE velocity that
takes 1 second to cross each $a$ tile and 2 seconds to cross each $b$ tile. Then 
$\mu$ will be cohomologous to $i_a + 2 i_b$, while $dx$ is cohomologous to 
$i_a + i_b$. Since $\mu$ is not cohomologous to a multiple of $dx$ plus something
asymptotically negligible, $\bar F$ cannot be conjugated to a uniform translation
on $\Omega_{nP}$. 

Note that if we do a shape change by $\mu$ to $\Omega_{nP}$, then we obtain a space
$\Omega_\mu$ of tilings with $L_a=1$ and $L_b=2$. On $\Omega_\mu$, the time it takes to
cross a tile is exactly equal to its length, and the flow $\bar F_t$ has 
been conjugated to $\Gamma_t$. However, $\Gamma_t$ on $\Omega_\mu$ is not conjugate
to any $\Gamma_{\rho t}$ on $\Omega_{nP}$.

To summarize, $\check H^1(\Omega_{Fib},\R) = \check H^1_{AN}(\Omega_{Fib},\R)
\oplus \R$, so every shape change on $\Omega_{Fib}$ is conjugate to a uniform dilation,
and every unidirectional flow on $\Omega_{Fib}$ with sPE velocity is conjugate to uniform translation
on $\Omega_{Fib}$ at some speed $\rho \ne 0$. However, $\check H^1(\Omega_{nP},\R) \ne 
\check H^1_{AN}(\Omega_{nP},\R) \oplus \R$, so there are shape changes on 
$\Omega_{nP}$ that change the translational dynamics. Every flow on $\Omega_{nP}$ with sPE velocity
can be conjugated to a uniform flow on {\em some} shape-changed space $\Omega_\mu$, but
that is different from a uniform flow on the original space $\Omega_{nP}$.
\end{example}

\begin{example}[A Denjoy-like construction] Denjoy \cite{Denjoy} constructed a circle map
with irrational rotation number that is semi-conjugate, but not conjugate, to a uniform
rotation. Here we adapt Denjoy's construction to get a map of tiling spaces, with irrational
rotation class, that is semi-conjugate but not conjugate to uniform translation. 

We first recall Denjoy's construction.
Start with a circle $S^1=\R/\Z$ of 
length 1, an irrational number $\alpha$, and a countable collection $\{I_n\}, n\in\Z$
of closed intervals, disjoint from each other and from $S^1$, 
whose total length is finite. Let $r: S^1 \to S^1$ be rotation by $\alpha$. Pick a point $x_0 \in S^1$ and let $x_n = r^n(x_0)$. 

Now construct a new circle $\bar S^1$ by replacing each point $x_n \in S^1$ by
the interval $I_n$. The total length of $\bar S^1$ is finite, and we can rescale to give it total length 1. We define a new map $f: \bar S^1 \to \bar S^1$ that maps each
$I_n$ homeomorphically (say, linearly) onto $I_{n+1}$, and that equals $r$ on the 
complement of the inserted intervals. Note that $f$ is {\em not} minimal, since the
orbit of a point in the complement of the inserted intervals does not contain any points in $I_n$, so its closure does not contain any points in the interior of 
$I_n$. 

Since $f$ is not minimal, it cannot be topologically conjugate to an irrational rotation. Since $f$ has no periodic points, it cannot be topologically conjugate to 
a rational rotation. Thus $f$ is not conjugate to any rotation. 
However, $f$ is {\em semi-}conjugate to a rotation by $\alpha$, since 
if $\pi: \bar S^1 \to S^1$ is the obvious projection, then $\pi \circ f = r \circ \pi$. 

We now transfer this construction to the category of 1-dimensional FLC tiling spaces. Let $\Omega$
be such a tiling space whose tiles all have integer length. (Say, all of length 1).
Then there is a map $\pi: \Omega \to S^1$ that sends each tiling to the common
location of all its vertices $\pmod 1$. This gives $\Omega$ the structure of a 
Cantor bundle over $\bar S^1$.\footnote{It is always possible to do a shape change to make
all the tiles have integer length, so {\em every} tiling space is homeomorphic to a 
Cantor bundle over a circle. Moreover, every higher dimensional minimal FLC 
tiling space is 
homeomorphic to a Cantor bundle over a torus. \cite{Sadun-Williams}}

Now pick an irrational $\alpha \in (0,1)$ and construct $f$ as before.  
Let $F_0: \R \to \R$ be a lift of $f$ such that, for all $x \in \R$, 
$x < F_0(x)  < x+1$ and $F_0(x+1)=F_0(x)+1$. We now define a 
self-homeomorphism $F: \Omega \to \Omega$. For each tiling $T \in \Omega$ {\em with
vertices at integer points}, and for each $x \in \R$, let $F(T-x)=T-F_0(x)$. 
Since $(T+n)-F_0(x+n)=T-F_0(x)$, $F$ is well-defined. $F$ is not conjugate
to a uniform translation, but the projection $\pi: \bar S^1 \to S^1$ lifts to 
a self-map of $\Omega$ that semi-conjugates $F$ to $\Gamma_\alpha$.
\end{example}

\begin{example}[A map with multiple rotation classes]

In Theorem \ref{thm-unique} we showed that if a map $F \in \F(\Omega)$ has an irrational
rotation class, then this class is unique. Here we show that a map on a tiling space can 
have multiple rotation classes, all of them rational. 

Pick an $\epsilon \in (0, 1/10)$ and consider a Fibonacci tiling space $\Omega$ with
tile lengths $L_a=1+\epsilon$ and $L_b = 1 - \phi\epsilon$, where 
$\phi=(1+\sqrt{5})/2$ is the golden mean. 
The length of $n$ consecutive tiles is
close to $n$, with the difference being bounded by $\phi^2\epsilon < 3\epsilon < 0.3$, but is never exactly $n$. 

On this space, let $F = \Gamma_1$. Manifestly, $dx$ 
is a rotation class for $F$, and is cohomologous to $(1+\epsilon) i_a + (1-\phi \epsilon)i_b$. Note that the $F$-orbit of a tiling with a vertex at 
the origin always has a vertex within 0.3 of the origin, so this orbit is not dense and 
$dx$ is rational. 

Pick an arbitrary $s \in (0,1)$ and let $\mu_s = (1+s\epsilon)i_a + (1-s\epsilon \phi)i_b$. 
If $x_1$ and $x_2$ are corresponding points in large patches, then they are separated by 
the lengths of $n$ consecutive tiles. Since $\int_{x_1}^{x_2} \mu_s = s(x_2-x_1) + (1-s)n$, 
the integral $\int_{x_1}^{x_2} \mu_s$ is strictly between $n$ and $x_2-x_1$, and therefore has the same integer part
(either $n$ or $n-1$) as $x_2-x_1$. Since the estimates that characterize a rotation form $\mu$ only
depend on whether or not $\int_{x_1}^{x_2} \mu$ is an integer, and on the integer part of that integral, $\mu_s$ is a rotation form, and since $[i_a]$ and $[i_b]$ are a basis for $H^1(\Omega_{fib}) = \Z^2$, $[\mu_s] \ne [dx]$ as rotation classes.

\end{example}

\begin{example}[Maps without rotation classes]

We now construct a family of tiling spaces, and maps on those tiling space,
meeting the usual technical assumptions, for which  
rotation forms do not exist. By the contrapositive to Proposition \ref{mu-exists}, these 
maps are not time-1 samplings of flows with sPE velocity, and in particular are not
conjugate to translations on $\Omega$ or on any shape change of $\Omega$.

Our examples are all fusion tilings \cite{Frank-Sadun14}. These are hierarchical tilings that
generalize substitution tilings. In our examples, 
there are three species of tiles, 
labeled $a$, $b$, and $c$, that group into two kinds of clusters, 
$A_1$ and $B_1$. The $A_1$ and $B_1$ clusters then group into larger clusters
$A_2$ and $B_2$, that group into larger clusters $A_3$ and $B_3$, and so on.
We will show that every sPE 1-form $\mu$ has the property that, 
for all sufficiently large  $j$, $\int_{A_j} \mu=\int_{B_j} \mu$. 
However, our $F$
will be such that the number of steps needed to cross an $A_j$ tile is
different from the number of steps needed to cross a $B_j$ tile. In fact, 
the difference in those crossing times can be made arbitrarily large! 
This implies that the arbitrary form $\mu$ is not a rotation form, so  
rotation forms (and rotation classes) for these maps do not exist. 

We begin our construction with the tiles $a$, $b$, and $c$. Let $c$ be 
a tile of length $L_c=1$, and let $a$ and $b$ have lengths greater than 1.
(For example, $L_a=L_b=\pi$.) Pick a smooth function $\phi$ on the disjoint union of
the three tiles
$a$, $b$, and $c$, such that $\phi(x)=1$ on the endpoints
of the three tiles, and such that for any two points $x_1 < x_2$ within the same tile,
$x_1 + \phi(x_1) < x_2 + \phi(x_2)$. This extends to a function on all the points 
of any tiling constructed from $a$, $b$, and $c$ tiles. We now define a map
\begin{equation} F: \Omega \to \Omega, \qquad F(T) = T - \phi(T(0)), \end{equation}
where $T(0)$ means the origin of the tiling $T$. 
That is, our displacement map $\Phi: \Omega \to \R$ is given by
$\Phi(T) = \phi(T(0))$, and the function $\phi_T$ is just the restriction of 
$\phi$ to the points of $T$ itself.

Now pick integers $n_2, n_3, n_4, \ldots$ that grow sufficiently rapidly. 
(The precise rate required will depend on $\phi$.) We construct our ``supertiles''
$A_j$, $B_j$ recursively, as follows:
\begin{eqnarray}
A_1 & = & ac, \\
B_1 & = & bc, \\
A_j & = & (A_{j-1} B_{j-1})^{n_j} \hbox{ if } j>1, \\
B_j & = & A_{j-1}^{n_j} B_{j-1}^{n_j} \hbox{ if } j>1.
\end{eqnarray}
That is, an $A_j$ is an $A_{j-1}$ followed by a $B_{j-1}$ followed by 
an $A_{j-1}$ followed by a $B_{j-1}$, etc., with each kind of $(j-1)$-supertile 
appearing $n_j$ times, while a $B_j$ consists of $n_j$ $A_{j-1}$'s followed
by $n_j$ $B_{j-1}$'s. Note that every $A_j$ and $B_j$ supertile begins with
an $A_{j-1}$ and ends with a $B_{j-1}$. Thus every $A_j$ and $B_j$ is preceded
by a $B_{j-1}$ and followed by an $A_{j-1}$. Also note that the lengths $|A_j|$ of
$A_j$ and $|B_j|$ of $B_j$ are equal.

Let $\alpha$ be an sPE 1-form with some radius $R$. Then there is a 
positive integer $j-2$ such that $|A_{j-2}|=|B_{j-2}|>R$. This implies that
for every $k$ with $k>j-2$,
$\int_{A_{k}} \alpha$ is the same for every $A_{k}$ cluster, regardless of
where that cluster sits in the tiling. Likewise, $\int_{B_{k}} \alpha$ is the
same for every $B_k$ cluster. But then, for every $k \ge j$,
\begin{equation}
\int _{A_k} \alpha = n_k \int_{A_{k-1}} \alpha + n_k \int_{B_{k-1}} \alpha= \int_{B_k} \alpha.
\end{equation}
That is, integrating sPE forms cannot distinguish between $A$ and $B$ supertiles
of sufficiently high degree. 

Next we consider the dynamics of $F$. Note that $F$ takes the left endpoint
of each $c$ tile to the right endpoint. The specific dynamics of $F$ on 
$a$ and $b$ induce maps $F_{1,a}: c \to c$ and $F_{1,b}: c \to c$. For $x \in c$, consider iterates of $x$ under $F$. If $c$ is followed by $ac$, then let $F_{1,a}(x)$ be the first return of $x$ in $c$. In this way $F_{1,a}:c \to c$. Identifying the endpoints of $c$ to form a circle, $F_{1,a}$ is a circle map, and hence has a rotation number $\rho_{1,a}$. Likewise, let $F_{1,b}$ be the first return on $c$ when followed by $bc$, and let the induced rotation number on $c$ be $\rho_{1,b}$. For definiteness, we take
both $\rho_{1,a}$ and $\rho_{1,b}$ to lie in $[0,1)$. We also let $\rho_{1,ab}$
be the rotation number of $F_{1,b} \circ F_{1,a}$. Note that 
$\rho_{1,ab}$ is generically not congruent to $\rho_{1,a} + \rho_{1,b}$ (mod 1).

The number of steps needed to cross an $A_1$ is always one of two 
consecutive integers $m_{a,1}$ and $m_{a,1}+1$. Similarly, the number of steps
needed to cross a $B_1$ is either $m_{b,1}$ or $m_{b,1}+1$. 
If $N$ is a large integer, then the number of steps needed
to cross $A_1^N$ is approximately $N(m_{1,a}+1 -\rho_{1,a})$, and the 
number of steps needed to cross $B_1^N$ is approximately 
$N(m_{1,b}+1 -\rho_{1,b})$. (In both cases, it takes at most $N(m+1)$ steps, but
we save a step every time the circle map crosses back from the end of $c$ to
the beginning, which happens a fraction $\rho$ of the time.) 

As a result, if $n_2$ is large, the number of steps needed 
to cross $B_2=A_1^{n_2}B_1^{n_2}$ is approximately (that is, within one of) 
$n_2(m_{1,a}+m_{1,b}+2- \rho_{1,a}-\rho_{1,b})$. Meanwhile,
the number of steps needed to cross $A_2=(A_1B_1)^{n_2}$ is approximately
$n_2(m_{1,a}+m_{1,b}+2 - \rho_{1,ab})$. (Depending on whether it is sometimes 
possible to cross $A_1B_1$ in $m_{1,a}+m_{2,a}$ steps, we may need to take 
$\rho_{1,ab}$ in $[1,2)$ for this to apply.) Since $\rho_{1,ab}$ is not equal to 
$\rho_{1,a}+\rho_{1,b} \pmod{1}$, if we take $n_2$ large enough we can get the
crossing times to differ by an arbitrarily large amount. 

The crossing of $A_2$ and $B_2$ generate their own return maps from the copy
of $c$ immediately preceding the supertile to the one at the end of the
supertile. Call these $F_{2,a}$ and $F_{2,b}$, with rotation numbers 
$\rho_{2,a}$ and $\rho_{2,b}$. As before, the rotation number of 
$F_{2,b} \circ F_{2,a}$, which we denote $\rho_{2,ab}$, is generically not
equal to $\rho_{2,a} + \rho_{2,b}$ (mod 1). By taking $n_3$ large enough, we
can make the crossing time for $B_3=A_2^{n_3} B_2^{n_3}$, which is approximately 
$n_3(m_{2,a} + m_{2,b} + 2 - \rho_{2,a}-\rho{2,b})$ differ by an arbitrarily
large amount from the crossing time for $A_3 = (A_2B_2)^{n_3}$, which is 
approximately $n_3(m_{2,a}+m_{2,b}+2 - \rho_{2,ab})$.  

We repeat the process, using the fact that generically 
$\rho_{j-1,ab} \ne \rho_{j-1,a} + \rho_{j-1,b}$ to pick $n_j$ big enough so
that the crossing times for $A_n$ and $B_n$ differ by more than 1. The end 
result is a tiling space $\Omega$ and a map $F$ on $\Omega$ with sPE displacement that has no
rotation class. 
\end{example}

\begin{example}[A map with no rotation number]
\label{norotationnumber}

Finally, we construct an example that has a rotation {\em class}, but that doesn't have a 
rotation {\em number}. We consider a fusion tiling with two tile types, $a$, and $b$, each of length 1. 
The fusion rule is 
\begin{equation} A_n = A_{n-1}^{10^n} B_{n-1}, \qquad B_n = A B_{n-1}^{10^n}, \end{equation}
where $A_0=a$ and $B_0=b$. The translation flow on the resulting tiling space $\Omega$ is minimal but
not uniquely ergodic. Rather, there are two ergodic measures, $\mu$ and $\nu$, that describe
the frequencies of patches in a high-order $A$-supertile (where roughly 90\% of the tiles, 99\%
of the 1-supertiles, and 99.9\% of the 2-supertiles are of type $A$), and the frequencies of patches in 
a high-order $B$-supertile (where only 10\% of the tiles, 1\% of the 1-supertiles and 0.1\% of the 
2-supertiles are of type $A$).

Let $v_0(x)$ be exactly 1 on every $a$ tile and 2 on every $b$ tiles. Convolve $v_0$ with a smooth bump function
of narrow support to get a smooth sPE velocity function $v(x)$. Let $F$ be the 
time-1 sampling of the flow by $v$. By the results of Section \ref{sec-flows}, $dx/v$ is a rotation form, and 
$[dx/v]$ is a rotation class. However, the limit 
\begin{equation}\label{rho-limits} \lim_{n \to \pm \infty} \frac{f_T^n(x)-x}{n} \end{equation}
depends on the reference tiling $T$ and may not even exist. 

If $T$ is in the support of $\mu$, then roughly 90\% of the tiles are type $a$, and take about 1 unit of time to
cross, while 10\% are type $b$, and take about half a unit of time to cross. The time needed to cross a distance $L$
averages to approximately
\begin{equation} 0.9L + \frac{0.1L}{2} =0.95L \end{equation}
so the distance traveled per unit time is about $1/0.95 \approx 1.05$. In this case, both limits in 
(\ref{rho-limits}), as $n \to \infty$ and as $n \to -\infty$, will be around 1.05. 

If $T$ is in the support of $\nu$, then only 10\% of the tiles are of type $a$, 90\% are of type $b$, and the
time needed to cross a distance $L$ is, on average, 
\begin{equation} 0.1L + \frac{0.9L}{2} =0.55L. \end{equation}
\end{example}
In this case both limits in (\ref{rho-limits}) will be around $1/0.55\approx 1.8$. If $T$ is neither in the support
of $\mu$ nor in the support of $\nu$, then the limits in (\ref{rho-limits}) may not exist, or may not agree. For 
instance, if $T$ consists of an infinite-order $A$-supertile for $x>0$ and an infinite-order $B$ supertile for $x<0$,
then the limit as $n \to +\infty$ will be around 1.05, while the limit as $n\to -\infty$ will be around 1.8. 

Although this map comes from a flow and is conjugate to a uniform translation on a tiling space 
$\Omega_{[dx/v]}$, the concepts of rotation number and $\rho$-boundedness simply make no sense on $\Omega$
itself and none of the results of \cite{aliste2010translation} apply. 
This example shows that the assumption of unique ergodicity in Theorem \ref{rho-exists}, and in
our main theorems, is essential.

\section{Conclusions and open problems.} In this paper we have studied rotation theory in
the category of FLC tiling spaces and maps with sPE displacement. In this setting we have shown that
\begin{enumerate}
    \item Instead of having a rotation number, we need to study the rotation
    {\em class} $[\mu] \in H^1(\Omega ,\R)$.
    \item This class doesn't always exist and isn't always unique. However, {\em irrational}
    rotational classes are unique, and maps where the rotation class does not exist do not
    come from time-1 samplings of flows with sPE velocity. 
    \item Unlike with circle maps, 
    there is a difference between coming from a flow with sPE velocity and being 
    conjugate to a uniform translation. If a map $F: \Omega \to \Omega$ comes from a flow with sPE velocity, then $F$ is conjugate to uniform translation on a different space $\Omega_\mu$,
    where $\Omega_\mu$ is obtained from $\Omega$ by doing a shape change by the rotation class $\mu$. Whether or not
    $F$ is conjugate to uniform translation on $\Omega$ itself, via a map that preserves path components, is purely a cohomological question.
    \item If a map $F \in \F(\Omega)$ has a rotation class $[\mu]$ that is irrational, then
    $F$ is semi-conjugate to uniform translation on $\Omega_\mu$. $F$ is then semi-conjugate to uniform
    translation on $\Omega$ itself, via a map that preserves path components, if and only if $[\mu] \in \R\, dx \oplus
    H^1_{AN}(\Omega, \R)$.
\end{enumerate}

One natural open question is whether a version of Denjoy's Theorem applies in this situation.
Is there a condition on the smoothness of $F$ that guarantees that the resulting semi-conjugacies are in 
fact conjugacies? We believe the answer to be ``yes'', but the usual approaches to proving Denjoy's Theorem 
for circle maps do not carry over to tiling spaces. A genuinely new approach is needed.

Another natural open question is to study what happens when we relax the assumptions that 
our maps have sPE displacement and that our flows have sPE velocity. For instance, we could ask when a map 
$F$ with sPE displacement on a tiling space $\Omega$ with FLC is the time-1 sampling of a
flow whose velocity function is not necessarily sPE. The trouble is that, if such a flow
existed, the resulting
form $\mu = dx/v$ would not be sPE, so doing a shape change by $\mu$ would change $\Omega$
into a space of tilings that no longer have FLC. This problem is not really approachable 
using the existing theory of FLC tilings. 

Relatively little is known about tilings with infinite local complexity, or ILC. 
(See \cite{Frank-SadunILC} for some recent results.) The \v Cech cohomology of the 
tiling space is well-defined, but does not directly correspond to a theory of PE forms. 
An alternate approach is to study the cohomology of wPE forms, without regard to \v Cech
cohomology, but this ``weak PE cohomology'' is infinitely generated even for simple tilings like the Fibonacci tiling. 

The upshot is that, once we leave the category of FLC tilings and of sPE functions, flows and maps, the problem stops reflecting the fairly rigid structure of tilings. 
Instead, it becomes a problem about ``matchbox manifolds'' is general. 
That is, about foliated spaces with 1-dimensional leaves. Such problems are interesting
and well worth studying, but require a whole new toolkit. 
\section*{Acknowledgements}
J.A.-P. acknowledges financial support from CONICYT FONDECYT REGULAR 1160975, CHILE. The authors thank Michael Baake, Alex Clark, Franz G\"ahler, Antoine Julien, Johannes Kellendonk, John Hunton,  Jamie Walton and Dan Rust for helpful discussions.  

\end{document}